\documentclass[a4paper, 10pt, twoside, notitlepage]{amsart}

\usepackage{amsmath,amscd}
\usepackage{amssymb}
\usepackage{amsthm}
\usepackage{comment}
\usepackage{graphicx, xcolor}

\usepackage{mathrsfs}
\usepackage[ocgcolorlinks,linkcolor=blue]{hyperref}

\usepackage{bm}
\usepackage{bbm}
\usepackage{url}

\newcommand{\LC}{\left(}
\newcommand{\RC}{\right)}

\theoremstyle{plain}
\newtheorem{thm}{Theorem}[section]
\newtheorem{prop}{Proposition}[section]
\newtheorem{lem}[prop]{Lemma}
\newtheorem{cor}[prop]{Corollary}

\newtheorem{defi}[prop]{Definition}
\newtheorem{rmk}[prop]{Remark}

\numberwithin{equation}{section}

\newcommand {\R} {\mathbb{R}} 
 \newcommand {\N} {\mathbb{N}}
 
\newcommand {\p} {\partial}

\newcommand {\supp} {\text{supp}}

\newcommand{\wt}{\widetilde}

\newcommand{\abs}[1]{\lvert #1 \rvert}          
\newcommand{\norm}[1]{\lVert #1 \rVert}         

\title[The Calder\'on problem for nonlocal parabolic operators]{The Calder\'on problem for nonlocal parabolic operators}

\author[C.-L. Lin]{Ching-Lung Lin}
\address{Department of Mathematics, National Cheng- Kung University, Tainan 701, Taiwan.}
\email{cllin2@mail.ncku.edu.tw}

\author[Y.-H. Lin]{Yi-Hsuan Lin}
\address{Department of Applied Mathematics, National Yang Ming Chiao Tung University, Hsinchu 30050, Taiwan}
\email{yihsuanlin3@gmail.com}

\author[G. Uhlmann]{Gunther Uhlmann}
\address{Department of Mathematics, University of Washington
	and Institute for Advanced Study, the Hong Kong University of Science
	and Technology}
\email{gunther@math.washington.edu}

\begin{document}
	
	\maketitle
	
	\begin{abstract}
	We investigate inverse problems in the determination of leading coefficients for nonlocal parabolic operators, by knowing the corresponding Cauchy data in the exterior space-time domain. The key contribution is that we reduce nonlocal parabolic inverse problems to the corresponding local inverse problems with the lateral boundary Cauchy data. In addition, we derive a new equation and offer a novel proof of the unique continuation property for this new equation. We also build both uniqueness and non-uniqueness results for both nonlocal isotropic and anisotropic parabolic Calder\'on problems, respectively.

		\medskip
		
		\noindent{\bf Keywords.} Calder\'on problem, Cauchy data, nonlocal parabolic operators, unique continuation property, global uniqueness, non-uniqueness.
		
		\noindent{\bf Mathematics Subject Classification (2020)}: 35B35, 35R11, 35R30
		
	\end{abstract}

	\tableofcontents

	\section{Introduction}\label{Sec 1}
	
 In this work, we study a nonlocal analogue of the Calder\'on problem for nonlocal parabolic operators. The mathematical formulation in this work is given as follows: Let $\Omega \subset \R^n$ be a bounded domain with Lipschitz boundary $\p \Omega$ for $n\geq 2$, and $T>0$ be a real number. Consider the parabolic equation 
 \begin{align}\label{local para}
 	\begin{cases}
 		  \mathcal{H}v=0 &\text{ in }\Omega_T:=(-T,T) \times \Omega ,\\
 		v(t,x)=f(t,x) &\text{ on }\Sigma_T:=(-T,T) \times \Sigma  , \\
 		v(-T,x)=0 &\text{ for }x\in \Omega,
 	\end{cases}
 \end{align}
where 
\begin{align}\label{local para op.}
	\mathcal{H}:=\p_t -\nabla \cdot (\sigma \nabla )
\end{align}
denotes the parabolic operators and $\Sigma:=\p \Omega$.
Consider the coefficient $\sigma(x)=\LC \sigma_{ik}(x)\RC_{1\leq i, k \leq n}$ to be a positive definite Lipschitz continuous matrix-valued function satisfying 
\begin{align}\label{ellipticity condition}
\begin{cases}
	\sigma_{ik}=\sigma_{ki}, \text{ for all }i,j=1,2,\ldots, n, \\
		c_0 |\xi|^2 \leq \displaystyle\sum_{i,k=1}^n \sigma_{ik}(x)\xi_i \xi_k \leq c_0^{-1}|\xi|^2 , \text{ for any } x \text{ and }  \xi=(\xi_1,\ldots, \xi_n)\in \R^n,\\
		\abs{\sigma(x)-\sigma(z)}\leq C_0|x-z|, \text{ for }x,z\in \R^n,
\end{cases}
\end{align} 
where $c_0\in (0,1)$ and $C_0>0$ are constants. Meanwhile, we also adapt the notation 
$$
B_T:=(-T,T) \times  B,
$$ 
for any $B\subset \R^n$.

It is known the well-posedness of \eqref{local para} always holds whenever $f$ satisfies suitable regularity assumptions (see Section \ref{Sec 2}). 
Once the well-posedness holds for certain equations, we can study inverse problems via either the \emph{Cauchy data} or the \emph{Dirichlet-to-Neumann} (DN) map.
In this work, we utilize the lateral boundary Cauchy data as our measurements, which is given by 
$$
\mathcal{C}_{\Sigma_T}\subset L^2(0,T;H^{1/2}(\Sigma_T))\times L^2(0,T;H^{-1/2}(\Sigma_T))
$$ with 
\begin{align}\label{local Cauchy}
	\mathcal{C}_{\Sigma_T}:=\left\{ \left. v_f \right|_{\Sigma_T}, \,  \left. \sum_{i,k=1}^n \sigma_{ik}\p_{x_k}v_f \nu_i \right|_{\Sigma_T} \right\},
\end{align}
where $v_f$ is a solution of \eqref{local para}, and $\nu =(\nu_1,\ldots, \nu_n)$ is the unit outer normal on $\Sigma$.
The classical Calder\'on problem for the space-time parabolic equation \eqref{local para} is to determine $\sigma$ by using the information $\Lambda_{\sigma}$ on $\Sigma_T$.

As a matter of fact, we are interested in the Calder\'on problem for nonlocal parabolic equations, which can be formulated as an initial exterior value problem. Throughout this work, we restrict the function $\sigma=\LC \sigma_{ik} \RC_{1\leq i,k\leq n}$ to be the $n\times n$ identity matrix $\mathbf{I}_n$ outside $\overline{\Omega}$, so that $\sigma$ still satisfies the condition \eqref{ellipticity condition} in $\R^n$. Given $s\in (0,1)$, consider 
 \begin{align}\label{nonlocal para}
	\begin{cases}
		\mathcal{H}^su=0 &\text{ in }\Omega_T\\
		u(t,x)=f(t,x) &\text{ in }(\Omega_e)_T, \\
		u(t,x)=0 &\text{ for }t\leq -T \  \text{ and }\  x\in \R^n,
	\end{cases}
\end{align}
where $\mathcal{H}$ is the parabolic operator given by \eqref{local para op.}, and 
\[
\Omega_e:=\R^n \setminus \overline{\Omega}
\]
stands for the exterior domain. Due to the definition $\mathcal{H}^s$ (see the rigorous definition of $\mathcal{H}^s$ in Section \ref{Sec 2}), we cannot only pose the initial condition for \eqref{nonlocal para}, but we require the whole past time information in order to make the equation \eqref{nonlocal para} well-defined.
In short, with suitable regularity assumptions for exterior data $f$, the well-posedness of \eqref{nonlocal para} holds (see Section \ref{Sec 2}).

Furthermore, we can formulate the Calder\'on problem for nonlocal parabolic equations as follows. Let $W\subset \Omega_e$ be an arbitrarily nonempty open set, and we define the corresponding exterior partial Cauchy data given by 
$$
\mathcal{C}_{W_T}\subset \LC \widetilde{\mathbf{H}}^s((\Omega_e)_T)\RC\times \LC\widetilde{\mathbf{H}}^s(W_T)\RC^\ast
$$ with 
\begin{align}\label{nonlocal Cauchy}
	\mathcal{C}_{W_T}:=\left\{ u|_{(\Omega_e)_T}, \, \left. \mathcal{H}^s u \right|_{W_T} \right\} ,
\end{align}
where $\widetilde{\mathbf{H}}^s((\Omega_e)_T)$ is a suitable function space which will be introduced in Section \ref{Sec 2}, and $\LC \widetilde{\mathbf{H}}^s(W_T)\RC^\ast$ denotes the dual space of $\mathbf{H}^s(W_T)$.
Our inverse problem is to ask whether can we determine $\sigma$ by using the corresponding exterior partial Cauchy data or not. In particular, we propose the following two inverse problems in space-time domain for both local and nonlocal parabolic equations:

\begin{itemize}
	\item[\textbf{(1)}] \textbf{Local Calder\'on's problem.}  Can one determine the coefficient $\sigma$ from the local Cauchy data \eqref{local Cauchy} of \eqref{local para}?

		\item[\textbf{(2)}] \textbf{Nonlocal Calder\'on's problem.}  Can one determine the coefficient $\sigma$ from the nonlocal Cauchy data \eqref{nonlocal Cauchy} of \eqref{nonlocal para}?
\end{itemize}
In this work, we will answer the above two questions, and describe the relations between nonlocal and local Calder\'on's problems for both nonlocal and local parabolic equations. We want to show that the above nonlocal Calder\'on problem \textbf{(2)} can be reduced to the local Calder\'on problem \textbf{(1)}, and new unique continuation/determination results are established in this work.

\medskip

\noindent $\bullet$ \textbf{Literature review.} The fractional Calder\'on problem was first proposed and solved in the work \cite{ghosh2016calder}, where the authors determined the zero order potential for the fractional Schr\"odinger equation by using the exterior partial Cauchy data. The main tools in the study of fractional inverse problems are based on the \emph{global unique continuation property} and the \emph{Runge approximation property}. Using these methods, many researchers have investigated inverse problems for fractional equations under various settings of mathematical models, such as \cite{BGU18,BKS2022calderon,covi2022global,CLL2017simultaneously,CL2019determining,cekic2020calderon,CMRU2022higher,feizmohammadi2021fractional,GLX,GRSU18,KLW2022calderon,kowinverse,ghosh2021non,harrach2017nonlocal-monotonicity,harrach2020monotonicity,LL2020inverse,LL2022inverse,lai2019global,RS20Calderon,LLR2019calder,lin2020monotonicity,QU2022calder,railo2022fractional,railo2022low,ruland2018exponential} and some references therein. In addition, several interesting properties for nonlocal parabolic operators have been studied in \cite{arya2022space,banerjee2018monotonicity,banerjee2021harnack,banerjee2022space}.

Meanwhile, the Calder\'on problem to determine the lower order coefficient for a fractional space-time parabolic equation has been considered by \cite{LLR2019calder} for constant coefficients and \cite{BKS2022calderon} for variable coefficients. More precisely, given $0<s<1$, consider the following fractional parabolic equation 
\begin{align*}
	\begin{cases}
		\mathcal{H}^s u +q u =0 &\text{ in }\Omega_T, \\
		u=f&\text{ in }(\Omega_e)_T,\\
		u(t,x)=0 &\text{ for }t\leq -T \ \text{ and }\ x\in \R^n,
	\end{cases}
\end{align*}
where $q=q(t,x)\in L^\infty(\Omega_T)$. It has been shown that one can determine zero order potential $q$ by using the exterior DN map.

 Before stating our main results, let us characterize our mathematical setups in the following.

\begin{itemize}
	\item[\textbf{(S)}] For $n\geq 2$, Let $\Omega \subset \R^n$ be a bounded open set with Lipschitz boundary $\p \Omega$, and $T>0$ be a number. Let $\sigma^{(j)}=\LC \sigma^{(j)}_{ik}\RC_{1\leq i,k\leq n}$ satisfy \eqref{ellipticity condition} in $\overline{\Omega}$, and further assume that $\sigma_{ik}^{(j)}(x)=\delta_{ik}$ to be the Kronecker delta, for $x\in \Omega_e$ and $j=1,2$. Consider $\mathcal{H}_j$ to be of the form \eqref{local para op.}, for $j=1,2$.  For $0<s<1$, let $W\subset \Omega_e$ be arbitrarily nonempty open subsets, and define the exterior partial Cauchy data by 
	$$
	\mathcal{C}^{(j)}_{W_T}=\left\{ u_j|_{(\Omega_e)_T}, \, \left. \LC \mathcal{H}_j\RC ^s u_j \right|_{W_T} \right\},
	$$
	where $u_j\in \mathbf{H}^{s}(\R^{n+1})$ is the solution of 
	\begin{align*}
		\begin{cases}
			\LC \mathcal{H}_j \RC^su_j=0 &\text{ in }\Omega_T\\
			u_j=f &\text{ in }(\Omega_e)_T, \\
			u_j(t,x)=0 &\text{ for }t\leq -T \  \text{ and }\  x\in \R^n,
		\end{cases}
	\end{align*}
	for $j=1,2$. Throughout this paper, we always assume the exterior Dirichlet data $f\in C^\infty_c\LC \LC \Omega_e\RC_T \RC$ for the sake of convenience. Moreover, we define the local Cauchy data as usual to be
	$$
	\mathcal{C}^{(j)}_{\Sigma_T}:=\left\{ \left. v_j \right|_{\Sigma_T}, \, \left. \sigma_j \p_{\nu} v_j \right|_{\Sigma_T} \right\},
	$$ 
	where $v_j$ is a solution of 
	 \begin{align*}
		\begin{cases}
			\mathcal{H}_jv_j=0 &\text{ in }\Omega_T,\\
			v_j=f &\text{ on }\Sigma_T, \\
			v_j(-T,x)=0 &\text{ for }x\in \Omega,
		\end{cases}
	\end{align*}
and we use the following notation 
\begin{align}\label{conormal}
	\sigma_j\left. \p_{\nu} v_j \right|_{\Sigma_T} :=\left. \sum_{i,k=1}^n \sigma^{(j)}_{ik}\p_{x_k}v_j \nu_i \right|_{\Sigma_T}
\end{align}
to denote the Neumann data, for $j=1,2$. Here $\nu=\LC \nu_1,\ldots, \nu_n \RC$ denotes the unit outer normal on $\Sigma$.
\end{itemize}

Then we are ready to state the first main theorem.
\begin{thm}\label{T1}
Adopting all statements and notations given in $\mathbf{(S)}$, suppose that the exterior partial Cauchy data 
\begin{align}\label{same nonlocal Cauchy data}
	\mathcal{C}^{(1)}_{W_T}=\mathcal{C}^{(2)}_{W_T},
\end{align}
 then the lateral boundary Cauchy data are the same that 
  $$
 \mathcal{C}^{(1)}_{\Sigma_T}=\mathcal{C}^{(2)}_{\Sigma_T}.
 $$
\end{thm}

Via the result of Theorem \ref{T1}, we are able to reduce the Calder\'on problem for nonlocal parabolic equations to the Calder\'on problem for local parabolic equations. Based on Theorem \ref{T1}, one can immediately obtain the following result.

\begin{cor}[Global uniqueness]\label{Cor: uniqueness}
	Adopting all statements and notations given in $\mathbf{(S)}$, let $\sigma_j$ be positive Lipschitz continuous scalar functions defined in $\R^n$ with $\sigma_j=1$ in $\Omega_e$. Suppose that the nonlocal Cauchy data 
	\begin{align*}
		\mathcal{C}^{(1)}_{W_T}=\mathcal{C}^{(2)}_{W_T},
	\end{align*}
	then 
	\[
	\sigma_1=\sigma_2 \text{ in }\Omega.
	\]
\end{cor}

Next, we are also interested in the case that the leading coefficient is a matrix-valued function. For the local case (i.e. $s=1$), the non-uniqueness result has been investigated by \cite{guenneau2012transformation}, and we recall the result as follows. Let $\sigma(x)=\LC\sigma_{ij}(x)\RC_{1\leq i, j\leq n}$ be a Lipschitz continuous matrix-valued function satisfying \eqref{ellipticity condition}. 
Let $\mathbf{F}:\overline{\Omega}\to \overline{\Omega}$ be a $C^\infty$ diffeomorphism with $\mathbf{F}|_{\p \Omega}=\mathbf{Id}$ (the identity map).
It is known that if $v(t,x)$ is a solution to 
\begin{align*}
	\p_t  v-\nabla \cdot (\sigma \nabla v) =0 \text{ for }(t,x)\in \Omega_T
\end{align*}
if and only if $\wt v(t,y):=v(t,\mathbf{F}^{-1}(y))$ is a solution to 
\begin{align*}
	\mathbf{F}_\ast 1(y)\p_t \wt v-\nabla \cdot (\mathbf{F}_\ast\sigma \nabla \wt v) =0 \text{ for }(t,y)\in \Omega_T,
\end{align*}
where $\mathbf{F}_\ast$ denotes the \emph{push-forward} as 
\begin{align*}
	\begin{cases}
		\mathbf{F}_\ast 1 (y) =\left. \frac{1}{\det (D\mathbf{F})(x)}\right|_{x=\mathbf{F}^{-1}(y)}, \\
		\mathbf{F}_\ast\sigma(y) = \left. \frac{D\mathbf{F}^T (x) \sigma (x)D\mathbf{F}(x)}{\det (D\mathbf{F})(x)}\right|_{x=\mathbf{F}^{-1}(y)} .
	\end{cases}
\end{align*}
Here $D\mathbf{F}$ stands for the (matrix) differential of $\mathbf{F}$ and $D\mathbf{F}^T$ is the transpose of $D\mathbf{F}$. Due to the fact that $\mathbf{F}|_{\p \Omega}=\mathbf{Id}$, one can see that the (lateral) Cauchy data are the same, i.e.,
\begin{align*}
	\mathcal{C}_{\Sigma_T}^{\sigma}:=\left\{ v|_{\Sigma_T}, \left. \sigma \p_\nu v \right|_{\Sigma_T}\right\}=\left\{ \wt v|_{\Sigma_T}, \left. \mathbf{F}_\ast \sigma \p_\nu \wt v \right|_{\Sigma_T}\right\}:=\mathcal{C}_{\Sigma_T}^{\mathbf{F}_\ast \sigma},
\end{align*}
which implies the non-uniqueness property holds for local parabolic operators.

Similar to the local case, our last main result in this paper is to show that non-uniqueness also holds for the nonlocal parabolic case.

\begin{thm}[Non-uniqueness]\label{T2}
	Adopting all statements and notations given in $\mathbf{(S)}$, let $W=W_1=W_2\Subset \Omega_e$ be an arbitrary nonempty open subset, and $\sigma_j$ be globally Lipschitz continuous matrix-valued function in $\R^n$ satisfying \eqref{ellipticity condition}. Suppose that the exterior Cauchy data 
	\begin{align*}
		\mathcal{C}^{(1)}_{W_T}=\mathcal{C}^{(2)}_{W_T},
	\end{align*}
	then there exists a Lipschitz invertible map $\mathbf{F}:\R^n\to \R^n$ with $\mathbf{F}:\overline{\Omega}\to \overline{\Omega}$ and $\mathbf{F}|_{\Omega_e}=\mathbf{Id}$ (the identity map) such that 
	\[
	\sigma_2=\mathbf{F}_\ast \sigma_1 \text{ in }\Omega,
	\]
	where $\mathbf{F}_\ast$ denotes the push-forward of the map $\mathbf{F}$ of the form 
	\[
	 \mathbf{F}_\ast\sigma_1(y) = \left. \frac{D\mathbf{F}^T (x) \sigma_1(x)D\mathbf{F}(x)}{\det (D\mathbf{F})(x)}\right|_{x=\mathbf{F}^{-1}(y)}.
	\]
\end{thm}

The paper is organized as follows. In Section \ref{Sec 2}, we recall the well-posedness for both local and nonlocal parabolic equations, and review several function spaces which are used in this work. We also provide a rigorous definition for the nonlocal parabolic operator $\mathcal{H}^s$ for $0<s<1$, which is defined via the evolutive heat semigroup. 
In Section \ref{Sec 3}, we derive a new equation, which plays an essential role in the study of this problem. Meanwhile, we show novel Carleman estimates in order to prove the unique continuation property for the new equation.
In Section \ref{Sec 4}, we prove Theorem \ref{T1}.
In particular, we demonstrate a fact that any solution of local parabolic equations can be approximated by solutions of some nonlocal parabolic equations. Finally, we show both global uniqueness and non-uniqueness results for nonlocal parabolic equations in Section \ref{Sec 5}.

\section{Preliminaries}\label{Sec 2}

In this section, we provide fundamental properties for both local and nonlocal parabolic equations.
We first review the definition of weak solutions and well-posedness for the local linear parabolic equation \eqref{local para}, which can be found in \cite[Chapter 7]{Evan}.

Consider the local parabolic equation \eqref{local para}, and consider a function $\wt f$ defined on $\overline{\Omega_T}$ such that $\wt f|_{\Sigma_T}=f$. Let $w:=u-f$ in $\overline{\Omega_T}$, then $w$ solves 
\begin{align}\label{local para zero BC}
 	\begin{cases}
 		  \mathcal{H}w=F &\text{ in }\Omega_T:=\Omega\times (-T,T),\\
 		w(t,x)=0 &\text{ on }\Sigma_T:=\Sigma  \times (-T,T), \\
 		w(-T,x)=g(x) &\text{ for }x\in \Omega,
 	\end{cases}
\end{align}
where $F=-\mathcal{H}\wt f$ and $g=-\wt f(-T,x)$. 
Define the bilinear form $B[w,\varphi;t]$ by 
\[
B[w,\varphi;t]:=\int_{\Omega}\sigma(x) \nabla_xw(x,t) \cdot \nabla_x\varphi(x)\, dx,
\]
for any $\varphi\in H^1_0(\Omega)$. Then we are able to define the concept of weak solutions of \eqref{local para zero BC}.

 \begin{defi}[Weak solutions]
 A function $w\in L^2(-T,T;H^1_0(\Omega))$ with $\p_t w\in L^2 (-T,T;H^{-1}(\Omega))$ is called a weak solution of the initial boundary value problem \eqref{local para zero BC} if $w$ satisfies the following conditions:
 \begin{itemize}
 	\item[(a)] $\int_{\Omega} \p_t w(t,x) \varphi (x)\, dx +B[w,\varphi;t]=\int_{\Omega}F(t,x)\varphi(x)\, dx$, for any $\varphi \in H^1_0(\Omega)$ and for almost every (a.e.) time $t\in [-T,T]$,
 	\item[(b)] $w(-T,x)=g(x)$.
 \end{itemize}
\end{defi}

With the definition of weak solutions at hand, we have the following result.

\begin{lem}[Well-posedness]\label{wellpose1}
	Let  $\sigma=\LC \sigma_{ik} \RC_{1\leq i , k\leq n}$ be a Lipschitz continuous matrix-valued function satisfying \eqref{ellipticity condition}.
	For  any  $g\in L^2(\Omega)$,  $ F\in  L^2(-T,T;L^2(\Omega))$, the parabolic equation \eqref{local para zero BC} admits a unique weak solution $w\in L^2(-T,T;H^1_0(\Omega))$. Moreover, 
	\begin{align*}
		&\max_{0\leq t\leq T}\norm{w}_{L^2(\Omega)}+\norm{w}_{L^2(-T,T;H^1_0(\Omega))}+ \norm{\p_t w}_{L^2(-T,T;H^{-1}(\Omega))}\\
		&\quad \leq C \LC \norm{F}_{L^2(-T,T;L^2(\Omega))}+ \norm{g}_{L^2(\Omega)}\RC,
	\end{align*}
for some constant $C>0$ depending on $\Omega$, $T$ and $\sigma$.
\end{lem}

Notice that given arbitrary lateral Dirichlet data $f\in L^2(-T,T;H^{3/2}(\Sigma))$, there exists $\wt f\in L^2(-T,T;H^2(\Omega))$ such that $\wt f=f$ on $\Sigma_T$ in the trace sense.

\subsection{The nonlocal parabolic operator $\mathcal{H}^s$}

The definition for the nonlocal parabolic operator $\mathcal{H}^s$ can be found in \cite{BDLCS2021harnack,BKS2022calderon}. In the rest of this paper, we adopt the notation 
\[
\mathcal{L}:=-\nabla \cdot (\sigma\nabla)
\]
to denote the second order elliptic operator of divergence form, where $\sigma=\LC \sigma_{ik}\RC_{1\leq i, k \leq n}$ is a matrix-valued function given via \eqref{ellipticity condition} in $\overline{\Omega}$, and we define $\sigma_{ik}=\delta_{ik}$ in $\Omega_e$, for $i,k=1,\ldots, n$. With this positive definite matrix-valued function $\sigma$ defined in the whole $\R^n$, we assume that the parabolic operator $\p_t+\mathcal{L}$ in $\R\times \R^n$ possesses a globally defined fundamental solution $p( x,z,\tau )$, which satisfies 
\[
\mathcal{P}_t 1(t,x)=\int_{\R^n} p(x,z,\tau)\, dz=1, \text{ for every }x\in \R^n \text{ and }\tau>0,
\]
where $\mathcal{P}_t$ stands for the heat semigroup.

Consider the following evolution semigroup 
\begin{align}\label{e-semni gp}
	\mathcal{P}^\mathcal{H}_\tau u(t,x):=\int_{\R^n} p(x,z,\tau)u(t-\tau,z)\, dz, \quad \text{ for }u\in \mathcal{S}(\R^{n+1}),
\end{align}
where $p(x,z,\tau)$ is the heat kernel corresponding to $\p_{\tau}+\mathcal{L}$ and $\mathcal{S}(\R^{n+1})$ denotes the Schwarz space.
In addition, the heat kernel  $p(x,z,\tau)$ satisfies 
\begin{align}\label{est-heat-kernel-sec2}
	C_1 \LC \frac{1}{4\pi \tau}\RC^{n/2}e^{-\frac{c_1|x-z|^2}{4\tau}} \leq p(x,z,\tau)\leq C_2 \LC \frac{1}{4\pi \tau}\RC^{n/2}e^{-\frac{c_2|x-z|^2}{4\tau}},
\end{align}
for $j=1,2$, for some positive constants $c_1,c_2,C_1$ and $C_2$. Noticing that $\left\{\mathcal{P}^\mathcal{H}_\tau  \right\}_{\tau \geq 0}$ is a strongly continuous contractive semigroup such that $\norm{\mathcal{P}^\mathcal{H}_\tau u-u}_{L^2(\R^{n+1})}=\mathcal{O}(\tau)$. We are able to give the explicit definition of $\mathcal{H}^s$.

\begin{defi}
	Given $s\in (0,1)$ and $u\in \mathcal{S}(\R^{n+1})$, the nonlocal parabolic operator $\mathcal{H}^s$ can be defined via the Balakrishnan formula (see \cite{BKS2022calderon}) as 
	\begin{align}\label{H^s}
		\mathcal{H}^s u(t,x):=-\frac{s}{\Gamma(1-s)} \int_0 ^\infty \LC \mathcal{P}^\mathcal{H}_\tau u(t,x)-u(t,x) \RC \frac{d\tau}{\tau^{1+s}}.
	\end{align}
\end{defi}

One may also use the definition from \cite{BDLCS2021harnack} to define the nonlocal parabolic operator $\mathcal{H}^s$.
In further, by using the Fourier transform with respect to the time-variable $t\in \R$, one can express $\mathcal{H}^s u$ in terms of the Fourier transform. It is known that the heat semigroup $\left\{P_t\right\}_{t\geq 0}$ can be written by spectral measures as an identity of gamma functions:
\begin{align}
	\mathcal{P}_t =\int_0^\infty e^{-\lambda t}\, dE_{\lambda} \quad \text{ and }\quad -\frac{s}{\Gamma(1-s)}\int_0^\infty \frac{e^{-(\lambda+\mathbf{i}\rho)t}-1}{\tau^{1+s}}\, d\tau=(\lambda+\mathbf{i}\rho)^s,
\end{align}
for $\lambda>0$ and $\rho \in \R$, where $\mathbf{i}=\sqrt{-1}$.
Consider the Fourier transform $\mathcal{F}_t$ of $\mathcal{P}_\tau^{\mathcal{H}}u$ with respect to the $t$-variable, then we have 
\[
\mathcal{F}_t\LC\mathcal{P}_\tau^{\mathcal{H}}u \RC (\rho,\xi)=e^{-\mathbf{i}\rho \tau}\mathcal{P}_\tau \LC \mathcal{F}_tu(\rho,\cdot) \RC (\xi),
\]
which infers that the Fourier analogue of the definition \eqref{H^s}
\begin{align*}
	\mathcal{F}_t\LC \mathcal{H}^s u \RC (\rho,\cdot)=& -\frac{s}{\Gamma(1-s)}\int_0^\infty \frac{1}{\tau^{1+s}}\int_0^\infty \LC e^{-(\lambda+\mathbf{i}\rho)\tau}-1\RC\, dE_{\lambda}\LC \mathcal{F}_tu(\rho,\cdot) \RC d\tau\\
	 =& \int_0^\infty (\lambda+\mathbf{i}\rho)^s \, dE_{\lambda} \LC\mathcal{F}_tu(\cdot,\rho)\RC.
\end{align*}

\subsection{Function spaces}

We next turn to define several function spaces.
By using previous discussion, for any $u\in \mathcal{S}(\R^{n+1})$, one can write 
\[
\norm{\mathcal{F}_t \LC \mathcal{H}^su\RC( \rho,\cdot)}_{L^2(\R^n)}=\int_0^\infty \abs{\lambda+\mathbf{i}\rho}^{2s}\, d \norm{E_{\lambda}\LC \mathcal{F}_tu( \rho,\cdot)\RC}^2,
\]
for $\rho \in \R$. With this relation at hand, we can define the space $\mathbf{H}^{2s}(\R^{n+1})$ to be the completion of $\mathcal{S}(\R^{n+1})$ under the norm
\begin{align}
	\norm{u}_{\mathbf{H}^{2s}(\R^{n+1})}=\LC \int_{\R}\int_0^\infty \LC1 +\abs{\lambda+\mathbf{i}\rho}^2\RC^s  d\norm{E_{\lambda}\LC \mathcal{F}_tu(\rho,\cdot) \RC}^2 \, d\rho \RC^{1/2}.
\end{align}
Now,  let $r\in \R$ and $\mathcal{O}\subset \R^{n+1}$ be an open set, then one can define 
\begin{align*}
	\mathbf{H}^r(\R^{n+1})=&\Big\{ \text{Completion of }\mathcal{S}(\R^{n+1}) \text{ with respect to the norm}: \\
	& \qquad \quad \int_{\R}\int_0^\infty \LC1 +\abs{\lambda+\mathbf{i}\rho}^2\RC^{r/2}  d\norm{E_{\lambda}\LC \mathcal{F}_tu( \rho,\cdot) \RC}^2 \, d\rho \Big\}, \\
	\mathbf{H}^r(\mathcal{O})=&\left\{ u|_{\mathcal{O}}: \, u\in \mathbf{H}^r(\R^{n+1}) \right\},\\
	\widetilde{\mathbf{H}}^r(\mathcal{O})=&\text{closure of }C^\infty_c (\mathcal{O}) \text{ in }\mathbf{H}^r(\R^{n+1}).
\end{align*}
Moreover, we define 
\begin{align}
	\norm{u}_{\mathbf{H}^r(\mathcal{O})}:=\inf \left\{ \norm{v}_{\mathbf{H}^r(\R^{n+1})}:\,  v|_{\mathcal{O}} =u\right\}.
\end{align}
We also denote the dual spaces
\[
\mathbf{H}^{-r}(\mathcal{O})=\LC\widetilde{\mathbf{H}}^r(\mathcal{O})\RC^\ast \quad \text{ and }\quad \widetilde{\mathbf{H}}^{-r}(\mathcal{O})=\LC\mathbf{H}^{-r}(\mathcal{O})\RC^\ast. 
\]
On the other hand, given $a \in \R$, one may consider the parabolic type fractional Sobolev space 
\begin{align*}
	\mathbb{H}^a(\R^{n+1}):=\left\{ u\in L^2(\R^{n+1}):\,  \LC |\xi|^2+\mathbf{i}\rho \RC^{a/2}\widehat{u}(\rho,\xi )\in L^2(\R^{n+1})\right\},
\end{align*}
where $\widehat{u}(\xi,\rho)=\int_{\R^{n+1}}e^{-\mathbf{i}(t,x)\cdot (\rho,\xi)}u(x,t)\, dtdx$ denotes the Fourier transform of $u$ with respect to both $t$ and $x$ variables.

Meanwhile, the graph norm of $\mathbb{H}^a$-functions is given by 
\begin{align}\label{space H}
	\norm{u}_{\mathbb{H}^a(\R^{n+1})}^2:=\int_{\R^{n+1}}\LC 1+\LC |\xi|^4+|\rho|^2 \RC^{1/2}\RC^{a/2} \widehat{u}(\rho,\xi)\, d\rho d\xi.
\end{align} 
In addition, one can express the space 
$$
\mathbb{H}^{a}(\R^{n+1})=\mathbb{H}^{a/2,a}(\R^{n+1}), 
$$
where the exponents $a/2$ and $a$ denote the (fractional) derivatives of time and space, respectively.
In particular, as $a=s\in (0,1)$, via the discussion in \cite{BKS2022calderon}, it is known that 
\begin{align}\label{space identification}
\mathbf{H}^s(\R^{n+1})=\mathbb{H}^s(\R^{n+1}), \text{ for }s\in (0,1),
\end{align}
and we denote 
\[
\mathbb{H}^s_E :=\left\{ u\in \mathbb{H}^s(\R^{n+1}):\, \supp(u)\subset E \right\},
\]
for any closed set $E\subset \R^{n+1}$.

\subsection{Initial exterior value problems}

In this section, let us consider the initial exterior value problem of \eqref{nonlocal para}. In order to study the well-posedness of the initial exterior value problem \eqref{nonlocal para}, as shown in \cite{BKS2022calderon,LLR2019calder}, one can consider the adjoint operator $\mathcal{H}^s_\ast$ of $\mathcal{H}^s$. More precisely,  $\mathcal{H}^s_\ast$ can be defined in terms of the spectral resolution via
\[
\mathcal{F}_t \LC  \mathcal{H}^s_\ast u\RC (\rho,\cdot )=\int_0^\infty (\lambda-\mathbf{i}\rho)^s \, dE_{\lambda}\LC \mathcal{F}_t u(\rho,\cdot )\RC,
\]
for $u\in \mathcal{S}(\R^{n+1})$. Furthermore, one has that $\mathcal{H}_\ast^s=\LC -\p_t +\mathcal{L} \RC^s$ for $s\in (0,1)$.

Next, for any $f,g\in \mathcal{S}(\R^{n+1})$, one can derive that 
\begin{align}\label{upper for pairing}
	\begin{split}
		\left\langle \mathcal{H}^s f , g \right\rangle_{\R^{n+1}} =&\left\langle \mathcal{H}^{s/2}f, \mathcal{H}^{s/2}_\ast g\right\rangle_{\R^{n+1}} =\left\langle f, \mathcal{H}^s_\ast g\right\rangle_{\R^{n+1}} \\
		=& \int_{\R}\int_0^\infty (\lambda+\mathbf{i}\rho)^s \, d\langle E_{\lambda} \mathcal{F}_t f, \overline{\mathcal{F}_t g}\rangle ( \rho , \cdot )\, d\rho \\
		\leq & C \norm{f}_{\mathbb{H}^s(\R^{n+1})}\norm{g}_{\mathbb{H}^s(\R^{n+1})},
	\end{split}
\end{align}
for some constant $C>0$ independent of $f$ and $g$. In view of \eqref{upper for pairing}, one has the mapping property $\mathcal{H}^s: \mathbb{H}^s(\R^{n+1})\to \mathbb{H}^{-s}(\R^{n+1})$, where $\mathbb{H}^{-s}(\R^{n+1})$ stands for the dual space of $\mathbb{H}^s(\R^{n+1})$.
In the rest of this paper, we adopt the notation $\langle \cdot, \cdot\rangle_{D}$ to denote the natural pairing between a function and its duality, where $D\subset \R^{n+1}$ is an arbitrary set. For instance, given $g\in \widetilde{H}^s(D)$, we can write 
\begin{align*}
	\left\langle  \mathcal{H}^s f, g \right\rangle_{D}=\left\langle  \mathcal{H}^s f, g \right\rangle_{\widetilde{H}^s(D)^\ast \times \widetilde{H}^s(D)},
\end{align*}
where $\widetilde{H}^s(D)^\ast$ stands for the dual space of $\widetilde{H}^s(D)$ for some set $D\subset \R^{n+1}$.

With these properties of $\mathcal{H}^s$ and $\mathcal{H}^s_\ast$ at hand, we can define the bilinear map $\mathbf{B}(\cdot, \cdot)$ on $\mathbb{H}^s(\R^{n+1})\times \mathbb{H}^s(\R^{n+1})$ via 
\begin{align*}
	\mathbf{B}(f,g):=\left\langle \mathcal{H}^{s/2}f,\mathcal{H}^{s/2}_\ast g  \right\rangle_{\R^{n+1}}.
\end{align*}
By \eqref{upper for pairing}, it is known that 
\[
\abs{\mathbf{B}(f,g)}\leq C \norm{f}_{\mathbb{H}^s(\R^{n+1})}\norm{g}_{\mathbb{H}^s(\R^{n+1})},
\]
for some constant $C>0$ independent of $f$ and $g$, which shows the boundedness of the bilinear form $\mathbf{B}(\cdot ,\cdot)$. On the other hand, for the coercive, one can consider the cutoff solution akin to \cite{LLR2019calder,BKS2022calderon} by considering 
\[
u_T(t,x):=u(t,x)\chi_{[-T,T]}(t),
\]
where $\chi_{[-T,T]}(t)=\begin{cases}
	1 & \text{ for }t\in [-T,T]\\
	0 &\text{ otherwise}
\end{cases}$ 
is the characteristic function. Moreover, $\chi_{[-T,T]}$ is also a multiplier in the fractional Sobolev space $H^a(\R^n)$, for $|a|\leq \frac{1}{2}$ (for example, see \cite[Theorem 11.4 in Chapter 1]{lions2012non}).

More precisely, the coercivity can be seen via 
\begin{align*}
	\left\langle \mathcal{H}^{s/2}f, \mathcal{H}^{s/2}_\ast f\right\rangle_{\R^{n+1}}=& \int_{\R}\int_0^\infty (\lambda+\mathbf{i}\rho)^s \, d\| E_{\lambda} \LC \mathcal{F}_t f\RC( \rho, \cdot)\|^2\, d\rho\\
	=&\int_{\R}\int_0^\infty |\lambda+\mathbf{i}\rho|^s\LC \cos(s\theta)+\mathbf{i}\sin(s\theta) \RC \, d\| E_{\lambda} \LC \mathcal{F}_t f\RC( \rho,\cdot)\|^2\, d\rho\\
	=&\int_{\R}\int_0^\infty |\lambda+\mathbf{i}\rho|^s \cos(s\theta) \, d\| E_{\lambda} \LC \mathcal{F}_t f\RC( \rho,\cdot)\|^2\, d\rho,
\end{align*}
where $\tan \theta =\frac{\rho}{\lambda}$ such that $\theta \in (-\pi/2,\pi/2)$ since $\lambda\geq 0$. Here we used that $\sin (s\theta)$ is an odd function so that the third identity holds in the preceding identities. Moreover, due the range of $\theta \in (-\pi/2,\pi/2)$ and $s\in (0,1)$, one can obtain that 
\[
\cos(s\theta)\geq \cos(s\pi/2):=C_s >0,
\]
so that 
\begin{align*}
\left\langle \mathcal{H}^{s/2}f, \mathcal{H}^{s/2}_\ast f\right\rangle_{\R^{n+1}}\geq  C_s\int_{\R}\int_0^\infty |\lambda+\mathbf{i}\rho|^s  \, d\| E_{\lambda} \LC \mathcal{F}_t u\RC( \rho,\cdot)\|^2\, d\rho \geq C \norm{f}_{L^2(\R^{n+1})}^2,
\end{align*}
for some constant $C>0$ independent of $f$. This proves the coercivity for the bilinear form $\mathbf{B}(\cdot, \cdot)$.

Moreover, as shown in \cite{BKS2022calderon,LLR2019calder}, the information of $u(t,x)|_{t>T}$ will not affect the behavior of the solution $u|_{\Omega_T}$, so we can define the weak solution of \eqref{nonlocal para} with the cutoff function.

\begin{defi}[Weak solutions]
	Let $\Omega\subset \R^n$ and $T>0$ be given as before. Given $F\in \LC \mathbb{H}^s_{\overline{\Omega_T}}\RC^{\ast}$ and $f \in \mathbf{H}^s((\Omega_e)_T)$. A function $u\in \mathbb{H}^s(\R^{n+1})$ is called a weak solution of 
	\begin{align}\label{wellposed-nonlocal}
		\begin{cases}
			\mathcal{H}^su=F &\text{ in }\Omega_T\\
			u(t,x)=f(t,x) &\text{ in }(\Omega_e)_T, \\
			u(t,x)=0 &\text{ for }t\leq -T \  \text{ and }\  x\in \R^n,
		\end{cases}
	\end{align}
if $v:=(u-f)_T\in \mathbb{H}^s_{\overline{\Omega_T}}$ and 
\[
\mathbf{B}(u,\phi)=\langle F,\phi \rangle_{\R^{n+1}}, \text{ for any }\phi \in \mathbb{H}^s_{\overline{\Omega_T}},
\]
or 
\[
\mathbf{B}(v,\phi)=\langle F-\mathcal{H}^s f,\phi \rangle_{\R^{n+1}},\text{ for any }\phi \in \mathbb{H}^s_{\overline{\Omega_T}}.
\]
\end{defi}

Now, the well-posedness of \eqref{wellposed-nonlocal} can be stated as follows.

\begin{prop}[Well-posedness]\label{Prop: wellposed}
	Let $\Omega\subset \R^n$ and $T>0$ be given as before. Given $F\in \LC \mathbb{H}^s_{\overline{\Omega_T}}\RC^{\ast}$ and $f \in \mathbf{H}^s((\Omega_e)_T)$. Then there exists a unique $u_T\in \mathbb{H}^s(\R^{n+1})$ with $(u-f)_T\in \mathbb{H}^s_{\overline{\Omega_T}}$ satisfying $\mathcal{H}^s u_T=F$ in $\Omega_T$, and 
	\[
	\norm{u_T}_{\mathbb{H}^s(\R^{n+1})}\leq C\LC \norm{F}_{(\mathbb{H}^s_{\overline{\Omega_T}})^\ast}+\norm{f}_{\mathbf{H}^s((\Omega_e)_T)} \RC,
	\]
	for some constant $C>0$ independent of $u$, $f$ and $F$.
\end{prop}

With boundedness and coercivity of $\mathbf{B}(\cdot ,\cdot )$ at hand, the proof of Proposition \ref{Prop: wellposed} is based on the Lax-Milgram theorem for the bilinear map $\mathbf{B}(\cdot, \cdot)$ (a similar trick as in \cite{BKS2022calderon,LLR2019calder}), so we skip the detailed proof.

In addition, once we obtain the well-posedness of \eqref{nonlocal para}, we are able to define the corresponding exterior Cauchy data (or Dirichlet-to-Neumann map) via the bilinear form $\mathbf{B}(\cdot,\cdot)$ (same relation has been investigated in the works \cite{BKS2022calderon,LLR2019calder}). In fact, given arbitrarily open sets $W_1,W_2\subset \Omega_e$, the nonlocal Cauchy data is given by 
\begin{align*}
		\mathcal{C}_{(W_1)_T,(W_2)_T}:=\left\{ u|_{(W_1)_T}, \, \left. \mathcal{H}^s u \right|_{(W_2)_T} \right\} \subset (\mathbf{H}^s((W_1)_T))\times (\mathbf{H}^s((W_2)_T))^\ast,
\end{align*}
where the adjoint space $(\mathbf{H}^s((W_2)_T))^\ast$ can be verified by $\mathbf{B}(\cdot, \cdot)$. 

\subsection{The extension problem for $\mathcal{H}^s$}

We now review the extension problem for $\mathcal{H}^s$, which is a degenerate parabolic equation. Given $u\in \mathbf{H}^s(\R^{n+1})$, let $U=U(t,x,y)$ be the solution of the Dirichlet problem in $\R^{n+2}_+:=\R^{n+1}\times (0,\infty)$ 
\begin{align}\label{degen para PDE}
	\begin{cases}
		\mathcal{L}_s U= y^{1-2s}\p_t U -\nabla_{x,y}\cdot \LC y^{1-2s}\wt \sigma(x)\nabla_{x,y}U\RC=0 &\text{ in }\R^{n+2}_+,\\
		U(t,x,0)=u(t,x) &\text{ on }\R^{n+1},
	\end{cases}
\end{align}
where 
\begin{align*}
	\wt \sigma(x)=\left( \begin{matrix}
		\sigma(x) & 0\\
	0 & 1 \end{matrix} \right)
\end{align*}
denotes $(n+1)\times (n+1)$ matrix.
It is known that \eqref{degen para PDE} can be viewed as the parabolic counterpart as the famous Caffarelli-Silvestre extension problem of the fractional Laplacian (see \cite{caffarelli2007extension}) for $\mathcal{H}^s$.

For any open set $\mathcal{D}\subset \R^{n+1}\times (0,\infty)$, we define the weighted Sobolev space 
\[
H^1(\mathcal{D};y^{1-2s}dtdxdy):=\left\{U: \, \norm{U}_{H^1(\mathcal{D};y^{1-2s}dtdxdy)}<\infty\right\},
\]
where
\[
\norm{U}^2_{H^1(\mathcal{D};y^{1-2s}dtdxdy)}:=\int_{\mathcal{D}} y^{1-2s}\LC |U|^2+|\nabla_xU|^2 + |\p_yU|^2 \RC dtdxdy.
\]
Then we have the following result.

\begin{prop}[Extension problem]\label{Prop: extension}
	Let $s\in (0,1)$ and $u\in \mathbb{H}^s(\R^{n+1})$, then there is a solution $U=U(t,x,y)$ of \eqref{degen para PDE} such that 
    \begin{itemize}
    	\item[(1)] $\displaystyle\lim_{y\to 0+} U(\cdot, \cdot, y)=u(\cdot, \cdot)$ in $\mathbb{H}^s(\R^{n+1})$,
    	\item[(2)] $\displaystyle\lim_{y\to 0+}\frac{2^{1-2s}\Gamma(s)}{\Gamma(1-s)} y^{1-2s}\p_y U(\cdot, \cdot,y)=\mathcal{H}^s u$ in $\mathbb{H}^{-s}(\R^{n+1})$,
        \item[(3)] $\norm{U}_{H^1 (\R^{n+1}\times (0,M);y^{1-2s}dxdtdy)}\leq C_M \norm{u}_{\mathbb{H}^s(\R^{n+1})}$, where $C_M>0$ is a constant depending on $M$, which is independent of $u$ and $U$. 
    \end{itemize}
\end{prop}

The proof of the above proposition was shown in \cite[Theorem 3.1]{BKS2022calderon}, so we omit the proof.

\section{The new equation and its properties}\label{Sec 3}

Recall that $\mathcal{H}_j:=\p_t +\mathcal{L}_j$ is a parabolic operator, where $\mathcal{L}_j=-\nabla \cdot (\sigma_j\nabla)$, and $\sigma_j$ is a matrix-valued function satisfying \eqref{ellipticity condition} in $\R^n$, such that $\sigma_j=\mathbf{I}_n$ in $\Omega_e$, for $j=1,2$. Due to the definition of $\mathcal{H}_j$, it is not hard to see that 
$\left. \mathcal{H}_1\right|_{(\Omega_e)_T}=\left. \mathcal{H}_2\right|_{(\Omega_e)_T}=\left. \LC \p_t -\Delta\RC \right|_{(\Omega_e)_T} $ is the heat operator.

\subsection{Basic properties of the new equation}
Given arbitrarily nonempty open sets $W_1,W_2$ in $\Omega_e$ and $0<s<1$, let $f\in \mathbf{H}^s((W_1)_T)$, and by utilizing the well-posedness of the initial exterior value problem, there exists a unique solution $u_j\in \mathbb{H}^s(\R^{n+1})$ of
	\begin{align}\label{nonlocal para sec 3}
	\begin{cases}
		\LC \mathcal{H}_j \RC^su_j=0 &\text{ in }\Omega_T\\
		u_j=f &\text{ in }(\Omega_e)_T, \\
		u_j(t,x)=0 &\text{ for }t\leq -T \  \text{ and }\  x\in \R^n,
	\end{cases}
\end{align}
for $j=1,2$. With the condition \eqref{same nonlocal Cauchy data} at hand, one can always assume that 
\begin{align}\label{same nonlocal para}
	\LC \mathcal{H}_1\RC^s u_1=\LC \mathcal{H}_2\RC^s u_2 \text{ in }(W_2)_T.
\end{align}
Notice that the global unique continuation property for nonlocal parabolic equation has been studied by \cite[Theorem 1.3]{LLR2019calder} and \cite[Theorem 1.3]{BKS2022calderon}, for constant coefficients and variable coefficients nonlocal parabolic operators, respectively. However, even given $u_1=u_2=f$ in $(\Omega_e)_T$, with the condition \eqref{same nonlocal para}, one cannot apply the global unique continuation property directly in this work.
Thus, we need to analyze the relation \eqref{same nonlocal para} in a more detailed way.

Let $p_j(x,z,\tau)$ be the heat kernel corresponding to $\p_{\tau}+\mathcal{L}_j$ in $\R^n \times \R$ for $j=1,2$, which was introduced in Section \ref{Sec 2}. By using the notation \eqref{e-semni gp}, we consider the function 
\begin{align}\label{conv-heat-def}
	\mathbf{U}_j(t,\tau,x):=\mathcal{P}_\tau^{\mathcal{H}_j}u_j(t,x)=\int_{\R^n}p_j(x,z,\tau)u_j(t-\tau,z)\,  dz,
\end{align}
 for $j=1,2$. Unlike the (nonlocal) elliptic case as in \cite{GU2021calder}, the function $\mathbf{U}_j$ defined by \eqref{conv-heat-def} is no longer a solution to any parabolic equation. As a matter of fact, the next lemma plays an essential role in our study.
 
\begin{lem}\label{Lem: Key}
	Let $u_j \in \mathbb{H}^s(\R^{n+1})$ be the solution of \eqref{nonlocal para sec 3} and $\mathbf{U}_j$ be the function defined by \eqref{conv-heat-def}, then $\mathbf{U}_j$ solves 
		\begin{align}\label{U_j equation}
			\begin{cases}
				\LC \p_t +\p_\tau \RC \mathbf{U}_j(t,\tau,x)+\mathcal{L}_j \mathbf{U}_j(t,\tau,x)=0, & \text{ for }(t,\tau,x)\in(-T,T)\times (0,\infty)\times \R^n,  \\
					\mathbf{U}_j(t,0,x)=u_j(t,x) &\text{ for }(t,x)\in (-T,T)\times \R^n,\\
						\mathbf{U}_j(-T,\tau,x)=0 &\text{ for }(\tau,x )\in  (0,\infty) \times \R^{n}.
			\end{cases}
		\end{align}
\end{lem}

\begin{proof}
 The  following arguments hold for $j=1,2$. With the definition \eqref{conv-heat-def} of $\mathbf{U}_j(x,t,\tau)$ at hand, a direct computation yields that 
 \begin{align}\label{comp 1}
 	\begin{split}
 		&\LC \p_\tau +\mathcal{L}_j \RC \mathbf{U}_j \\
 		= &\int_{\R^n } \left[\LC \p_\tau +\mathcal{L}_j \RC p_j(x,z,\tau)  \right] u_j (t-\tau, z)\, dz \\
 		& +\int_{\R^n}p_j(x,z,\tau)\p_\tau \LC u_j (t-\tau,z)\RC dz \\
 		=&\int_{\R^n}p_j(x,z,\tau)\p_\tau \LC u_j (t-\tau,z)\RC dz , 
 	\end{split}
 \end{align}
for $(t,\tau,x)\in (-T,T) \times (0,\infty) \times \R$, where we used that $p_j(x,z,\tau)$ is the heat kernel of $\p_{\tau}+\mathcal{L}_j$, for $j=1,2$. 
By interchanging the derivatives of $\tau$ and $t$, the right hand side of \eqref{comp 1} can be rewritten as 
\begin{align*}
	\begin{split}
		\LC \p_\tau +\mathcal{L}_j \RC \mathbf{U}_j =&-\p_t\LC \int_{\R^n}p_j(x,z,\tau) u_j (t-\tau,z) \, dz\RC\\
		=&-\p_t \mathbf{U}_j \quad \text{ for }(t,\tau ,x)\in  (-T,T)\times (0,\infty) \times \R^n, 
	\end{split}
\end{align*}
which shows the first equation of \eqref{U_j equation} holds.
Meanwhile, it is not hard to see that 
\begin{align*}
	\mathbf{U}_j(t,0,x)=&\lim_{\tau \to 0}\mathcal{P}^{\mathcal{H}_j}_\tau  u_j(t,x)\\
	=&\lim_{\tau \to 0}\int_{\R^n}p_j(x,z,\tau)u_j(t-\tau ,z)\,  dz=u_j(t,x), \text{ for }(t,x)\in (-T,T) \times \R^n.
\end{align*}
Finally, since the parameter $\tau \in (0,\infty)$, one can directly find that 
\begin{align*}
	\mathbf{U}_j(-T,\tau,x)=\int_{\R^n}p_j(x,z,\tau)u_j(-T-\tau,z)\,  dz=0, \text{ for }(\tau,x)\in  (0,\infty) \times \R^n,
\end{align*}
where we used $u_j(-T-\tau,x)=0$ for $\tau >0$ and $j=1,2$. This proves the assertion.
\end{proof}

\begin{lem}\label{Lem: energy estimate}
	Adopting all notations in Lemma \ref{Lem: Key}, we have 
	\begin{align}\label{energy 0}
		\begin{split}
			\max_{-T\leq t\leq T}\int_0^{\infty} \int_{\R^n} |\mathbf{U}_j|^2\, dxd\tau+\int_{-T}^{T}\int_0^{\infty }\int_{\R^n}|\nabla \mathbf{U}_j|^2 \,dxd\tau dt  \leq C\norm{u_j}_{\mathbb{H}^s(\R^{n+1})},
		\end{split}
	\end{align}
 for some constant $C>0$ independent of $\mathbf{U}_j$ and $u_j$, for $j=1,2$.
\end{lem}

\begin{proof}
	Multiplying \eqref{U_j equation} by $\mathbf{U}_j$, an integration by parts with respect to the $x$-variable yields that 
	\begin{align}\label{energy 1}
		\frac{\p_t+\p_\tau}{2}\LC \int_{\R^n}|\mathbf{U}_j|^2\, dx \RC+ \int_{\R^n}\sigma_j \nabla \mathbf{U}_j \cdot \nabla \mathbf{U}_j \,dx=0.
	\end{align}
    We next integrate \eqref{energy 1} with respect to both $t$ and $\tau$ variables, which gives rise to 
    \begin{align}\label{energy 2}
    \begin{split}
    	0=&\int_0^{\infty} \int_{\R^n} |\mathbf{U}_j|^2(\wt t, \tau ,x )\, dxd\tau -	\int_0^{\infty} \int_{\R^n} |\mathbf{U}_j|^2(-T, \tau ,x )\, dxd\tau  \\
    	& + \left[\int_{-T}^{\wt t} \int_{\R^n} |\mathbf{U}_j|^2( t, \tau ,x )\, dxdt \right]_{\tau=0}^{\tau=\infty}  \\	
    	& +2\int_{-T}^{\wt t}\int_0^{\infty }\int_{\R^n}\sigma_j \nabla \mathbf{U}_j \cdot \nabla\mathbf{U}_j \,dxd\tau dt  \\
    	=&\int_0^{\infty} \int_{\R^n} |\mathbf{U}_j|^2(\wt t, \tau ,x )\, dxd\tau -\int_{-T}^{\wt t} \int_{\R^n} |u_j|^2( t,x)\, dxdt\\
    	& +2\int_{-T}^{\wt t}\int_0^{\infty }\int_{\R^n}\sigma_j \nabla \mathbf{U}_j \cdot\nabla \mathbf{U}_j \,dxd\tau dt 
    \end{split}
    \end{align}
    for any $\wt t\in (-T,T)$, where we used 
    \[
    \lim_{\tau \to \infty}\mathbf{U}_j(t,\tau,x)=0
    \]
    from theheat kernel estimate \eqref{est-heat-kernel-sec2}. By rewriting \eqref{energy 2}, we have 
    \begin{align}\label{energy 3}
   	\begin{split}
   		\int_0^{\infty} \int_{\R^n} |\mathbf{U}_j|^2(\wt t, \tau ,x )\, dxd\tau+2\int_{-T}^{\wt t}\int_0^{\infty }\int_{\R^n}\sigma_j \nabla \mathbf{U}_j \cdot \nabla \mathbf{U}_j \,dxd\tau dt  \leq \norm{u}_{\mathbb{H}^s(\R^{n+1})}.
   	\end{split}
   \end{align}
    Combined with the ellipticity of $\sigma_j$, the inequality \eqref{energy 0} holds.
\end{proof}

With Lemma \ref{Lem: energy estimate} at hand, we immediately obtain the following result.

\begin{cor}\label{Cor: unique solution}
	The equation \eqref{U_j equation} possesses a unique solution.
\end{cor}

\begin{proof}
	If there are two solutions with the same initials $\mathbf{U}_j(t,0,x)$ and $\mathbf{U}_j(-T,\tau,x)$,  then the right hand side of \eqref{energy 0} is zero. Therefore, the solution is unique.
\end{proof}

\subsection{The Carleman estimate}\label{Sec: Carleman}

The proof of main theorem are based on suitable Carleman estimates. In the rest of this section, we will derive the needed Carleman estimates. In fact, our aim is to derive Carleman estimates with the weight 
$$
\varphi_{\beta}=\varphi_{\beta}(x) =\exp(\psi(y)),
$$ where $\beta>0$, $y=-\log|x|$ and $\psi(y)=\beta y+\frac{1}{16}\beta e^{-y/2}$.
From \cite[Appendix]{KLW2016doubling}, $\psi(y)$ is a convex function satisfying
\begin{equation}\label{1.1}
	\left\{
	\begin{aligned}
		&\frac{1}{2}\beta\le \psi'\le \beta,\\
		&\mbox{dist}(2\psi',{\mathbb Z})+\psi''\geq \frac{1}{32}.
	\end{aligned}
	\right.
\end{equation}
Further,  $h$ satisfies that for any $C>0$ there exists $R_0>0$ such that
\begin{equation}\label{1.2}
	\frac{1}{16}|x|\beta\le \LC1+\psi''(-\ln |x|)\RC
\end{equation}
for all $\beta$ and $|x|\le R_0$.

We will modify the arguments of \cite[Lemma 2.1]{LW2022quantitative}.
First, let us introduce polar
coordinates in ${\mathbb R}^n \backslash {\{0\}}$ by writing $x=r
\omega$, with $r=|x|$, $\omega=(\omega_1,\cdots,\omega_n)\in
\mathbb S^{n-1}$. Using new coordinate $y=-\log r$, we obtain that
$$
\frac{\partial}{\partial x_j}=e^{y}\LC -\omega_j \partial_y +\Omega_j\RC,\quad 1\le j\le n,
$$
where $\Omega_j$ is a vector field in $\mathbb S^{n-1}$. We could check that
the vector fields $\Omega_j$ satisfy
$$\sum_{j=1}^n\omega_j\Omega_j=0\quad\text{and}\quad\sum_{j=1}^n\Omega_j\omega_j=n-1.$$
Since $r\rightarrow 0$ if and only if $y\rightarrow {\infty}$, we are interested in values of $y$ near $\infty$.

It is easy to see that
\begin{equation*}
	\frac{\partial ^2}{\partial x_j \partial x_{\ell}}=e^{2y}\LC -\omega_j
	\partial_y -\omega_j +\Omega_j\RC \LC -\omega_{\ell} \partial_y +\Omega_{\ell}\RC,\quad 1\le j,\ell\le n.
\end{equation*}
then the Laplacian becomes
\begin{equation*}
	e^{-2y}\Delta =\partial^2_y -(n-2)\partial_y +\Delta_\omega,
\end{equation*}
where $\Delta_\omega=\sum^n_{j=1}\Omega^2_j$ denotes the
Laplace-Beltrami operator on $\mathbb S^{n-1}$. Let us recall that the
eigenvalues of $-\Delta_\omega$ are $k(k+n-2)$, $k\in \mathbb{N}$, and denote
the corresponding eigenspaces are $E_k$, where $E_k$ is the space of
spherical harmonics of degree $k$.  We note that
\begin{equation}\label{1.3}
	\sum_{j}\iint |\Omega_j v|^2 dy d\omega=\sum_{k\geq 0}k(k+n-2)\int|
	v_k |^2 dy,
\end{equation}
where $v_k$ is the projection of $v$ onto $E_k$. Let
$$
\Lambda=\sqrt{\frac{(n-2)^2}{4}-\Delta_{\omega}},
$$
then $\Lambda$ is an elliptic first-order positive
pseudodifferential operator in $L^2(\mathbb S^{n-1})$. The eigenvalues of
$\Lambda$ are $k+\frac{n-2}{2}$ and the corresponding eigenspaces
are $E_k$ which represents the space of spherical harmonics of
degree $k$. Hence

\begin{align}\label{1.4}
	\Lambda=\sum_{k\geq 0}\LC  k+\frac{n-2}{2}\RC \pi_{k},
\end{align}
where $\pi_k$ is the orthogonal projector on $E_k$. Let
$$
L^{\pm}=\partial_y-\frac{n-2}{2}\pm\Lambda,
$$
then it follows that
\begin{equation*}
	e^{-2y}\Delta =L^+L^-=L^-L^+.
\end{equation*}
Denote $L_\beta^{\pm}=\partial_y-\frac{n-2}{2}\pm\Lambda-\psi'(y) $. Then we have that
$L_\beta^{\pm}v=e^{\psi(y)} L^{\pm} \LC e^{-\psi(y)}v\RC$ and
$e^{-2y}e^{\psi(y)}\Delta \LC  e^{-\psi(y)}v\RC=L_\beta^{+}L_\beta^{-}v=L_\beta^{-}L_\beta^{+}v$.

\begin{lem}\label{lem2.1}
	Let $\chi(t), \zeta(\tau)\in C_0^{2}(\mathbb R)$.
	There are sufficiently large constants $\beta_1$, depending on $n$,  such that for all $v(t,\tau,y,\omega)\in C^1(\R^2;C^{\infty}({\mathbb R}\times \mathbb S^{n-1}))$ and $\beta\geq \beta_1$
	with $\beta\in \mathbb{N}+\frac{1}{4}$, we have that
	\begin{align}\label{1.5}
		\begin{split}
				&\int |\chi \zeta(L_\beta^{+}L_\beta^{-}v-e^{-2y}\partial_tv-e^{-2y}\partial_\tau v)|^2+\int |\chi' \zeta e^{-2y}v|^2+\int |\chi \zeta' e^{-2y}v|^2\\
			&\qquad \gtrsim \sum_{j+|\alpha|\leq1}\beta^{2-2(j+|\alpha|)}\int (1+\psi'')|\partial_y^j\Omega^\alpha (\chi \zeta v)|^2,
		\end{split}
	\end{align}
	where  ${\rm supp} v(t,\tau,y,\omega)\subset \R^2\times (0,\infty)\times \mathbb S^{n-1}$.
\end{lem}

\begin{proof}
	By diagonalizing $v=\sum_kv_k$ and $L_\beta^{+}L_\beta^{-}v=(\partial_y-\psi'+k)(\partial_y-\psi'-k-n+2)v_k$,  it is enough to prove that
\begin{align*}
		&\sum_{j\leq1}\int (1+\psi'')|(\beta^{2-2j}+k^{2-2j})\partial_y^j (\chi\zeta v)|^2\\
		&\qquad \lesssim\int|\chi \zeta (L_\beta^{+}L_\beta^{-}v-e^{-2y}\partial_tv-e^{-2y}\partial_\tau v)|^2+\int |\chi' \zeta e^{-2y}v|^2+\int |\chi \zeta' e^{-2y}v|^2,
\end{align*}
where we abuse the notation $v=v_k$.
By direct computations, we can have that
\begin{align}\label{1.6}
	\chi\zeta  (\partial_y-\psi'+k)(\partial_y-\psi'-k-n+2)v=\chi\zeta( \partial_y^2v-\tilde{b}\partial_yv+\tilde{a}v),
\end{align}
where
\begin{align*}
	\begin{cases}
		\tilde{a}=(\psi'-k)(\psi'+k+n-2)-\psi''\\
		\tilde{b}=2\psi'+n-2.
	\end{cases}
\end{align*}
It is helpful to note that
\[
\psi'=\beta-\frac{\beta}{32}e^{-y/2},\quad\psi''=\frac{\beta}{64}e^{-y/2},\quad\psi''=-\frac{\beta}{128}e^{-y/2}.
\]
We obtain from \eqref{1.6} that
\begin{align}\label{1.7}
	\begin{split}
		&4|\chi\zeta(L_\beta^{+}L_\beta^{-}v-e^{-2y}\partial_tv-e^{-2y}\partial_\tau v)|^2+4|\chi'\zeta e^{-2y}v|^2+4|\chi \zeta' e^{-2y}v|^2\\
		\geq&|\chi\zeta( \partial_y^2v-\tilde{b}\partial_yv+\tilde{a}v)- e^{-2y}\partial_t(\chi\zeta v)- e^{-2y}\partial_\tau (\chi\zeta v)|^2\\
		=&|H(v)|^2-2\tilde{b}\partial_y(\chi\zeta v)H(v)-2 e^{-2y}\partial_t(\chi\zeta v)H(v)-2 e^{-2y}\partial_\tau (\chi\zeta v)H(v)\\
		&+|\tilde{b}\partial_y(\chi\zeta v)+ e^{-2y}\partial_t(\chi\zeta v)+ e^{-2y}\partial_\tau (\chi\zeta v)|^2,
	\end{split}
\end{align}
where $H(v):=\chi\zeta \LC\partial_y^2v+\tilde{a}v\RC$. Now we write
\begin{equation}\label{1.8}
	\left\{
	\begin{aligned}
		&-2\int \tilde{b}\partial_y(\chi\zeta v)H(v) =-2\int \tilde{b}\partial_y(\chi\zeta v)\partial_y^2(\chi\zeta v)-2\int \tilde{a}\tilde{b}\chi\zeta v\partial_y(\chi\zeta v)\\
		&-2\int e^{-2y}\partial_t(\chi\zeta v)H(v) =
		-2\int e^{-2y}\partial_t(\chi\zeta v)\partial_y^2(\chi\zeta v)-2\int \tilde{a}\chi\zeta ve^{-2y}\partial_t(\chi\zeta v)\\
		&-2\int e^{-2y}\partial_\tau(\chi\zeta v)H(v) =
		-2\int e^{-2y}\partial_\tau(\chi\zeta v)\partial_y^2(\chi\zeta v)-2\int \tilde{a}\chi\zeta ve^{-2y}\partial_\tau(\chi\zeta v).\\
	\end{aligned}\right.
\end{equation}
Direct computations imply that
\begin{equation}\label{1.9}
	\left\{
	\begin{aligned}
		&-2\int \tilde{b}\partial_y(\chi\zeta v)\partial_y^2(\chi\zeta v)=2\int\psi''|\partial_y(\chi\zeta v)|^2,\\
		&-2\int \tilde{a}\tilde{b}\chi\zeta v\partial_y(\chi\zeta v)=\int\partial_y(\tilde{a}\tilde{b})|\chi\zeta v|^2,
	\end{aligned}\right.
\end{equation}

\begin{equation}\label{1.10}
	-2\int e^{-2y}\partial_t(\chi\zeta v)\partial_y^2(\chi\zeta v)=-4\int e^{-2y}\partial_t(\chi\zeta v)\partial_y(\chi\zeta v),
\end{equation}
\begin{equation}\label{1.11}
	\begin{aligned}
		-2\int \tilde{a}\chi\zeta ve^{-2y}\partial_t(\chi\zeta v)=0,
	\end{aligned}
\end{equation}

\begin{equation}\label{1.12}
	-2\int e^{-2y}\partial_\tau(\chi\zeta v)\partial_y^2(\zeta v)=-4\int e^{-2y}\partial_t(\chi\zeta v)\partial_y(\chi\zeta v),
\end{equation}
\begin{equation}\label{1.13}
	\begin{aligned}
		-2\int \tilde{a}\chi\zeta ve^{-2y}\partial_\tau(\chi\zeta v)=0.
	\end{aligned}
\end{equation}

Note that here $\tilde{a}$ is independent of $t, \tau$. Combining \eqref{1.7} to \eqref{1.13} yields
\begin{eqnarray}\label{1.14}
	&& \int |\chi\zeta( \partial_y^2v-\tilde{b}\partial_yv+\tilde{a}v)- e^{-2y}\partial_t(\chi\zeta v)- e^{-2y}\partial_\tau (\chi\zeta v)|^2\notag\\
	&\geq&\int\left( |H(v)|^2+|\tilde{b}\partial_y(\chi\zeta v)+ e^{-2y}\partial_t(\chi\zeta v)+ e^{-2y}\partial_\tau (\chi\zeta v)|^2\right) \notag\\
	&&+2\int \psi''|\partial_y(\chi\zeta v)|^2-4\int e^{-2y}\partial_t(\chi\zeta v)\partial_y(\chi\zeta v)-4\int e^{-2y}\partial_\tau(\chi\zeta v)\partial_y(\chi\zeta v)\notag\\
	&&+\frac{17}{3}\int (\psi')^2\psi''|\chi\zeta v|^2-2\int (k^2+nk-2k)\psi''|\chi\zeta v|^2
\end{eqnarray}
for $\beta\ge\beta_1$. It is helpful to remark that $\psi''>0$.

Likewise, we write
\begin{align}\label{1.15}
	\begin{cases}
			\quad |\tilde{b}\partial_y(\chi\zeta v)+ e^{-2y}\partial_t(\chi\zeta v)+ e^{-2y}\partial_\tau (\chi\zeta v)|^2\\
			=|(\tilde{b}-2)\partial_y(\chi\zeta v)+e^{-2y}\partial_t(\chi\zeta v)+e^{-2y}\partial_\tau (\chi\zeta v)+2\partial_y(\chi\zeta v)|^2\\
			=|(2\psi'+n-4)\partial_y(\chi\zeta v)+e^{-2y}\partial_t(\chi\zeta v)+e^{-2y}\partial_\tau (\chi\zeta v)|^2\\
			\quad+4(2\psi'+n-3)|\partial_y(\chi\zeta v)|^2 +4e^{-2y}\partial_t(\chi\zeta v)\partial_y(\chi\zeta v)+4 e^{-2y}\partial_\tau(\chi\zeta v)\partial_y(\chi\zeta v).\\
			\frac{1}{2}|H(v)|^2=\frac{1}{2}|H(v)+ 3\psi''\chi\zeta v|^2-3\psi''\chi\zeta vH(v)-\frac{9}{2}(\psi'')^{2}|\chi\zeta v|^2.\\
	\end{cases}
\end{align}
It is easy to check that
\begin{align}\label{1.16}
	&\quad\;-3\int \psi''\chi\zeta vH(v)\notag\\
	&=-3\int \psi''\chi^2\zeta^2 v(\partial_y^2v+\tilde{a}v)\notag\\
	&\ge 3\int \psi''|\partial_y(\chi\zeta v)|^2-\frac{10}{3}\int (\psi')^2\psi''|\chi\zeta v|^2+3\int (k^2+nk-2k)\psi''|(\chi\zeta v)|^2
\end{align}
for all $\beta\ge\beta_1$. Moreover, via \eqref{1.14}-\eqref{1.16}, we have that for $\beta\ge\beta_1$
\begin{align}\label{1.17}
	\begin{split}
		& \int |\chi\zeta( \partial_y^2v-\tilde{b}\partial_yv+\tilde{a}v)- e^{-2y}\partial_t(\chi\zeta v)- e^{-2y}\partial_\tau (\chi\zeta v)|^2\\
		\ge&\,8\int \psi'|\partial_y(\chi\zeta v)|^2+2\int (\psi')^2\psi''|\chi\zeta v|^2+\int k^2\psi''|(\chi\zeta v)|^2\\
		&+\frac{1}{2}\int|H(v)|^2.
	\end{split}
\end{align}
Now, we write that
\begin{align*}
		\frac{1}{4}\int|H(v)|^2=&\frac{1}{4}\int  \left|H(v)-\frac{\beta(\psi'-k)\chi\zeta v}{10|\beta-k|}\right|^2+\int\frac{\beta(\psi'-k)}{20|\beta-k|}\chi\zeta vH(v)\\
		&-\int\frac{\beta^2(\psi'-k)^2}{400|\beta-k|}|\chi\zeta v|^2\\
		\geq&\int\frac{\beta(\psi'-k)}{20|\beta-k|}\chi\zeta vH(v)-\int\frac{\beta^2(\psi'-k)^2}{400|\beta-k|}|\chi\zeta v|^2
\end{align*}
and note
\begin{align*}
		\int\frac{\beta(\psi'-k)}{20|\beta-k|}\chi^2\zeta^2 v\partial_y^2v=&-\int\frac{\beta(\psi'-k)}{20|\beta-k|}|\partial_y(\chi\zeta v)|^2+\int\frac{\beta\psi'''}{40|\beta-k|}|\chi\zeta v|^2\\
		=&-\int\frac{\beta(\beta-k)}{20|\beta-k|}|\partial_y(\chi\zeta v)|^2+\int\frac{\beta^2 e^{-y/2}}{640|\beta-k|}|\partial_y(\chi\zeta v)|^2\\
		&+\int\frac{\beta\psi'''}{40|\beta-k|}|\chi\zeta v|^2
\end{align*}
with
\begin{align*}
		&\int\frac{\beta(\psi'-k)}{20|\beta-k|}\chi^2\zeta^2 vav\\
		=&\int\frac{\beta(\psi'-k)^2(\psi'+k+n-2)}{20|\beta-k|}|\chi\zeta v|^2+\int\frac{\beta(\psi'-k)\psi''}{20|\beta-k|}|\chi\zeta v|^2.
\end{align*}
Combining \eqref{1.17}, we have that
\begin{align}\label{1.18}
	\begin{split}
		& \int \left|\chi\zeta( \partial_y^2v-\tilde{b}\partial_yv+\tilde{a}v)- e^{-2y}\partial_t(\chi\zeta v)- e^{-2y}\partial_\tau (\chi\zeta v)\right|^2\\
		\ge&\,7\int \psi'|\partial_y(\chi\zeta v)|^2+\frac{3}{2}\int (\psi')^2\psi''|\chi\zeta v|^2+\frac{1}{2}\int k^2\psi''|(\chi\zeta v)|^2\\
		&+\int\frac{\beta(\psi'-k)^2(\psi'+k+n-2)}{20|\beta-k|}|\chi\zeta v|^2+\frac{1}{4}\int|H(v)|^2.
	\end{split}
\end{align}
Thus, we can get the desire estimate if $\beta\geq\beta_1$.
\end{proof}

By Lemma \ref{lem2.1}, we have our main Carleman estimate.
\begin{lem}[Carleman estimate]\label{lem2.2}
	Let $\chi(t), \zeta(\tau)\in C_0^{2}(\mathbb R)$.
	There is a sufficiently large number $\beta_2$ depending on $n$ such that
	for all $w(t,\tau,x)\in C^1(\R^2;C^{\infty}({\mathbb R}^n))$ and
	$\beta\geq \beta_2$ with $\beta\in \mathbb{N}+\frac{1}{4}$, we have that
	\begin{eqnarray}\label{1.19}
		&&\iiint\varphi^2_\beta(1+\psi'')\chi^2\zeta^2 \LC |x|^{-n+2}|\nabla ( w)|^2+\beta^2|x|^{-n}| w|^2 \RC dxd\tau dt\notag\\
		&\lesssim& \iiint \varphi^2_\beta|x|^{-n+4}\chi^2\zeta^2\LC \Delta w-\partial_t w-\partial_\tau w\RC^2dxd\tau dt\notag\\
		&&+\iiint \varphi^2_\beta|x|^{-n+4}|\chi'\zeta w|^2\, dxd\tau dt+\iiint \varphi^2_\beta|x|^{-n+4}|\chi\zeta' w|^2\, dxd\tau dt,
	\end{eqnarray}
	where  ${\rm supp} (w(t,\tau,x))\subset \R\times (0,\infty)\times \left\{x: \, |x|<e \right\}$.
\end{lem}

\subsection{Unique continuation property}\label{Sec: UCP}

This section is devoted to proving the unique continuation property of solutions to
\begin{eqnarray}\label{2.1}
	\partial_t u+\partial_\tau u-\Delta u=0.
\end{eqnarray}
The arguments are motivated by the proof of \cite[Theorem 15]{vessella2003carleman}.
\begin{thm}\label{Thm: UCP (new)}
	Let $u\in H^1(\R;H^1((0,\infty));H^2(\R ^n))$ be a nontrivial solution of \eqref{2.1}.
	Given $t_0, \tau_0, \hat{\tau}>0$ such that $t<T$ and $\tau_0<\hat{\tau}/2$.
	Assume that $u(t,\tau,x)=0$ in $\{(t,\tau,x): \|x\|<R_1, 0< \tau<\hat{\tau}, |t|<T\}$.
	Then $u(t,\tau,x)=0$ in $\{(t,\tau,x): x\in \mathbb{R}^n, 0< \tau<\hat{\tau}, |t|<T\}$.
\end{thm}

\begin{proof}
Let $\chi$ be defined as
\begin{equation}\label{2.2}
	\chi (t)=
	\begin{cases}
			1,& |t|\leq T_2,\\
			0,&|t|\geq T_1,\\
			\exp\left(-(\frac{T}{T_1-|t|})^{3}(\frac{|t|-T_2}{T_1-T_2})^{4}\right),& T_2<|t|<T_1,
	\end{cases}
\end{equation}
where $T_1=T-\frac{t_0}{2}$, $T_2=T-t_0$.

Similarly, we define $\zeta$ as
\begin{equation}\label{2.3}
	\zeta (\tau)=
	\begin{cases}
			1,& |\tau-\hat{\tau}/2|\leq \tau_2,\\
			0,& |\tau-\hat{\tau}/2|\geq \tau_1,\\
			\exp\left(-(\frac{\hat{\tau}}{8(\tau_1-|\hat{\tau}/2|)})^{3}(\frac{|\hat{\tau}/2|-\tau_2}{\tau_1-\tau_2})^{4}\right),& \tau_2<|\tau-\hat{\tau}/2|<\tau_1,
	\end{cases}
\end{equation}
where $\tau_1=\frac{\hat{\tau}}{2}-\frac{\tau_0}{2}$, $\tau_2=\frac{\hat{\tau}}{2}-\tau_0$.

Moreover, we let $\theta(x)\in C^{\infty}_0
({\mathbb R}^n)$ satisfy $0\le\theta(x)\leq 1$ and
\begin{equation*}
	\theta (x)=
	\begin{cases}
			1, & |x|<R_2,\\
			0,& |x|>2R_2,
	\end{cases}
\end{equation*}
where $R_1<R_2<R_0/2$. It is easy to see that for any multiindex
$\alpha$
\begin{equation}\label{2.4}
	|D^{\alpha}\theta|=O(|x|^{|\alpha|})\quad \text{if}\quad R_2 <|x|< 2R_2.
\end{equation}
Applying \eqref{1.19} to $\theta u$ gives
\begin{align}\label{2.5}
	&\int\varphi^2_\beta(1+\psi'')\chi^2\zeta^2 |x|^{-n} (|x|^{2}|\nabla ( \theta u)|^2+\beta^2| \theta u|^2)\notag\\
	\lesssim& \int \varphi^2_\beta|x|^{-n+4}\chi^2\zeta^2(\Delta (\theta u)-\partial_t (\theta u )-\partial_\tau (\theta u ))^2\notag\\
	&+\int \varphi^2_\beta|x|^{-n+4}\theta^2\zeta^2|\chi' u|^2+\int \varphi^2_\beta|x|^{-n+4}\chi^2\theta^2|\zeta' u|^2,
\end{align}
Here and after, $C$ and $\tilde C$ denote general constants whose value may vary from line to line. The dependence of $C$ and $\tilde C$ will be specified whenever necessary.

By using \eqref{2.4} and \eqref{2.1}, we obtain that
\begin{align}\label{2.6}
	\begin{split}
		&\int_{W_{T,\hat{\tau}}}\varphi^2_\beta(1+\psi'')\chi^2\zeta^2 |x|^{-n} (|x|^{2}|\nabla  u|^2+\beta^2| u|^2)\\
		\lesssim&\int\varphi^2_\beta(1+\psi'')\chi^2\zeta^2 (|x|^{-n+2}|\nabla ( \theta u)|^2+\beta^2|x|^{-n}| \theta u|^2)\\
		\lesssim& \int \varphi^2_\beta|x|^{-n+4}\chi^2\zeta^2(\Delta (\theta u)-\partial_t (\theta u )-\partial_\tau (\theta u ))^2\\
		&+\int \varphi^2_\beta|x|^{-n+4}\theta^2\zeta^2|\chi' u|^2+\int \varphi^2_\beta|x|^{-n+4}\chi^2\theta^2|\zeta' u|^2\\
		\leq & \int_{\tilde{W}} \varphi^2_\beta|x|^{-n+4}\zeta^2|\chi' u|^2+\int_{\tilde{W}} \varphi^2_\beta|x|^{-n+4}\chi^2|\zeta' u|^2+\int_{\tilde{Y}}\varphi^2_\beta|x|^{-n}|\tilde{U}|^2,
	\end{split}
\end{align}
where $W_{T,\hat{\tau}}=\{(t,\tau,x):\, |t|<T,\, 0<\tau< \hat{\tau},\, |x|<R_2\}$,
$\tilde{W}=\{(t,\tau,x):\, |t|<T,\, 0<\tau< \hat{\tau},\, |x|<R_2\}$, $\tilde{Y}=\{(t,\tau,x):\, |t|<T,\, 0<\tau< \hat{\tau},\, R_2<|x|<2R_2\}$, and
$|\tilde{U}(x)|^2=|x|^{4}|\chi' u|^2+|x|^{4}|\zeta' u|^2+|x|^{-2}|\nabla u|^2+|x|^{-4}|u|^2$.
Here, the same terms on the right hand side of \eqref{2.6} are absorbed by the left hand
side of \eqref{2.6}.  With the choices described above, we obtain from \eqref{2.6} that
\begin{align}\label{2.7}
	\begin{split}
		&\int_{W_{T,\hat{\tau}}}\varphi^2_\beta(1+\psi'')\chi^2\zeta^2 |x|^{-n}\beta^2| u|^2\\
		\leq  &\tilde{J}_1+\tilde{J}_2+C\int_{\tilde{Y}}\varphi^2_\beta|x|^{-n}|\tilde{U}|^2,
	\end{split}
\end{align}
where $$
\begin{aligned}
	\tilde{J}_1&=C\int_{\tilde{W}} \varphi^2_\beta|x|^{-n+4}\zeta^2\left|\frac{\chi'}{\chi}\right|^2|\chi u|^2,\\
	\tilde{J}_2&=C\int_{\tilde{W}} \varphi^2_\beta|x|^{-n+4}\chi^2\left|\frac{\zeta'}{\zeta} \right|^2|\zeta u|^2.
\end{aligned}
$$
Notice that we define $\frac{\chi'}{\chi}=0$ as $\chi=0$. The arguments for estimating $\tilde{J}_1$ and $\tilde{J}_2$ are the same, so we only estimate $\tilde J_1$..
To do so, one only needs to consider the integral over $\tilde{W}_1=\{(t,\tau,x):\, T_2<|t|<T_1,\, 0<\tau< \hat{\tau},\, |x|<R_2\}$.
To this end, we consider the following two cases. Firstly,
$$C\left|\frac{\chi'}{\chi}\right|^2\leq \frac{\beta^3}{4}|x|^{-3}\leq \frac{(1+\psi'')\beta^2}{4}|x|^{-4}.$$
In this case, $\tilde{J}_1$ will be absorbed by the left hand side. Secondly, we consider
$$C\left|\frac{\chi'}{\chi}\right|^2\geq \frac{\beta^3}{4}|x|^{-3}.$$
Since
\begin{equation*}
	\sqrt{C}\left|\frac{\chi'}{\chi}\right|\leq C_1\frac{T^3}{(T_1-|t|)^4},
\end{equation*}
we can consider a large set
\begin{eqnarray}\label{2.8}
	C_1\frac{T^3}{(T_1-|t|)^4}\geq \LC\frac{\beta^3}{4|x|^{3}}\RC^{1/2}.
\end{eqnarray}
As a result, taking $\beta\geq \beta_3$ with $\beta_3=\left(\frac{C_1^24^9T^6}{t_0^8}\right)^{1/3}\geq\left(\frac{C_1^24^9R_2^3T^6}{t_0^8}\right)^{1/3}$,
we can get from \eqref{2.8} that
$$T_1-|t|\leq t_0/4$$
which implies that
\begin{eqnarray}\label{2.9}
	|t|-T_2\geq \frac{T_1-T_2}{2}.
\end{eqnarray}
Combining \eqref{2.1}, \eqref{2.8} and \eqref{2.9}, we get for $(t,x)\in \tilde{W}_1$ that
\begin{eqnarray}\label{2.10}
	\chi(t)\lesssim\exp\left(-\frac{1}{16}\LC\frac{\beta^3T^2}{4|x|^{3}}\RC^{3/8}\right).
\end{eqnarray}
Thus, we have from \eqref{2.10} and \eqref{2.1} to obtain for $\beta\geq \beta_3$ that
\begin{eqnarray}\label{2.11}
	\tilde{J}_1\leq C_2\int_{\tilde{W}}|u|^2,
\end{eqnarray}
where $C_2$ is a positive constant depending on $\lambda,n, T, t_0$.

Combining \eqref{2.7} and \eqref{2.11}, we get that
\begin{align}\label{2.12}
	\begin{split}
		&\beta^2(R_2)^{-n}\varphi^2_\beta(R_2)\int_{W_2}| u|^2\\
		\leq &\int_{W_{T,\hat{\tau}}}\varphi^2_\beta(1+\psi'')\chi^2\zeta^2 |x|^{-n}\beta^2|u|^2\\
		\lesssim &\int_{\tilde{W}}|u|^2+(R_2)^{-n}\varphi^2_\beta(R_2)\int_{\tilde{Y}}|\tilde{U}|^2,
	\end{split}
\end{align}
where $W_2=\{(t,\tau,x):\, |t|<T-t_0,\, \tau_0<\tau< \hat{\tau}-\tau_0,\, |x|<R_2\}$.

Dividing $\beta^2(R_2)^{-n}\varphi^2_\beta(R_2)$ on both sides of \eqref{2.11} and if $\beta\geq n$, we have
\begin{eqnarray}\label{2.13}
	&\int_{W_2}| u|^2&\lesssim \beta^{-2}(R_2)^{n}\varphi^{-2}_\beta(R_2)\int_{\tilde{W}}|u|^2+\beta^{-2}\int_{\tilde{Y}}|\tilde{U}|^2.
\end{eqnarray}
Let $\beta\rightarrow \infty$ on \eqref{2.13}, we get that $u=0$ on $W_2$.

Since $t_0$ and $\tau_0$ are arbitrary, we derive that $u=0$ on $\{(t,\tau,x):|t|<T,\, \tau< \hat{\tau},\, |x|<R_2\}$.
By the standard argument, we can obtain that $u=0$ on $\{(t,\tau,x):\, |t|<T,\, 0<\tau< \hat{\tau},\, x \in \mathbb{R}^n\}$.
Finally, since $\hat{\tau}$ can be arbitrary, we have that $u=0$ on $\{(t,\tau,x):\,  |t|< T ,\, \tau\in (0,\infty),\, |x|\in \mathbb{R}^n\}$.
\end{proof}

\begin{cor}\label{Cor: UCP}
	Given an open set $D \subset \R^n$ and $T>0$. Let $\mathcal{O}\subset D$ be a subset. Let $u$ be a solution of $\LC \p_t +\p_\tau -\Delta\RC u =0$ for $(x,t,\tau)\in D\times (-T,T)\times (0,\infty)$. If $u(x,t,\tau)=0$ in $\mathcal{O}\times (-T,T)\times (0,\infty)$, then $u=0$ in $ D\times (-T,T)\times (0,\infty)$.
\end{cor}

\section{From the nonlocal to the local}\label{Sec 4}

In this section, let us discuss several useful materials and prove Theorem \ref{T1}.

\subsection{Auxiliary tools and regularity results}

\begin{lem}
	Consider the function 
	\begin{align}\label{V_j}
		\mathbf{V}_j(t,x):=\int_0^\infty \mathbf{U}_j(t,\tau,x)\, d\tau, 
	\end{align}
	then $\mathbf{V}_j$ is the solution of 
	\begin{align}\label{V_j equation}
		\begin{cases}
			\mathcal{H}_j\mathbf{V}_j=u_j &\text{ in }\LC \R^{n}\RC_T,\\
			\mathbf{V}_j(-T,x)=0 &\text{ on }\R^n,
		\end{cases}
	\end{align}
	where $u_j\in \mathbb{H}^s(\R^{n+1})$ is the solution of \eqref{nonlocal para sec 3}, for $j=1,2$.
\end{lem}

\begin{proof}
	The  following arguments hold for $j=1,2$. Integrating \eqref{U_j equation} with respect to the $\tau$-variable, one has 
	\begin{align}\label{comp 2}
		\begin{split}
			\LC \p_t +\mathcal{L}_j\RC \mathbf{V}_j(t,x)=&\int_{0}^\infty \LC \p_t +\mathcal{L}_j\RC \mathbf{U}_j(t,\tau ,x)\, d\tau\\=&-\int_{0}^\infty\p_\tau \LC \mathbf{U}_j(t,\tau ,x)\RC d\tau\\
			=&\mathbf{U}_j(t,0,x), \quad \text{ for }(t,x)\in \R^{n+1},
		\end{split}
	\end{align}
	for $j=1,2$.
	and plug the above relation into \eqref{comp 2}, so that \eqref{V_j equation} holds.
	Finally, one can check that 
	\begin{align*}
		\mathbf{V}_j(-T,x)=\int_0^\infty \int_{\R^n}p_j(x,z,\tau)u_j(-T-\tau,z)\,  dz  d\tau=0,
	\end{align*}
	for $t\in (-T,\infty)$, and $\tau \in (0,\infty)$, where we utilized that $u_j(x,t)=0$ for $t\leq -T$ and $x\in \R^n$ (or $u_j(z,-T-\tau)=0$ for $\tau \geq 0$ and $z\in \R^n$). This proves the assertion.
\end{proof}

	From the above derivation, it is not hard to see that 
	\begin{align}\label{past time V_j}
		\mathbf{V}_j(\zeta,x)=0, \text{ for all }\zeta\leq -T,
	\end{align}
   which will be used in the forthcoming discussion. We next analyze the regularity result of the solution $\mathbf{V}_j$.

   \begin{lem}[Regularity estimate]\label{Lem: regularity}
   	The function $\mathbf{V}_j$ given by \eqref{V_j} satisfies 
   	\begin{align*}
   	\begin{split}
   			 \LC \mathcal{H}_j\RC^{s/2} \mathbf{V}_j \in L^2(-T,T; H^{2}(\R^n))  \quad \text{ and } \quad  \p_t 		\LC \mathcal{H}_j\RC^{s/2} \mathbf{V}_j  \in   L^2 ((\R^n)_T),
   	\end{split}
   	\end{align*}
    where $H^{a}(\R^n)$ denotes the (fractional) Sobolev space of order $a\in \R$, for $j=1,2$.
   \end{lem}

\begin{proof}
	 Since $u_j\in \mathbb{H}^s(\R^{n+1})$ is the solution of \eqref{nonlocal para sec 3}, it is not hard to check that 
	\[
	\wt u_j:= \LC \mathcal{H}_j\RC^{s/2} u_j \in L^2(\R^{n+1}),
	\]
	for $j=1,2$.
	Here we used the known result that $\LC \mathcal{H}_j\RC^{s/2}: \mathbb{H}^s(\R^{n+1})\to L^2(\R^{n+1})$ by observing the Fourier symbol in the definition \eqref{space H} of the function space $\mathbb{H}^s(\R^{n+1})$.
	
	Let us consider the function 
	$$\widetilde{\mathbf{V}}_j:= \LC \mathcal{H}_j\RC^{s/2}\mathbf{V}_j
	$$ 
	where $\mathbf{V}_j$ satisfies \eqref{V_j equation}. Note that  
	\begin{align*}
		\widetilde{\mathbf{V}}_j(-T,x)= &\LC \LC \mathcal{H}_j\RC^{s/2}\mathbf{V}_j \RC (-T,x) \\
		=&-\frac{s/2}{\Gamma(1-s/2)}\int_0^\infty \frac{\mathcal{P}^{\mathcal{H}_j}_\tau \mathbf{V}_j(-T,x)-\mathbf{V}_j(-T,x)}{\tau^{1+s/2}}\, d\tau \\
		=&0, 
	\end{align*}
for $x\in \R^n$, where we used \eqref{past time V_j}. Then $\widetilde{\mathbf{V}}_j$ is the solution of 
	\begin{align}\label{wt V_j equation}
		\begin{cases}
			\mathcal{H}_j\widetilde{\mathbf{V}}_j=\wt u_j &\text{ in }\R^n_T,\\
			\widetilde{\mathbf{V}}_j(-T,x)=0 &\text{ on }\R^n,
		\end{cases}
	\end{align}
    where we used the interchangeable property between $\LC \mathcal{H}_j\RC^{s/2}$ and $\mathcal{H}_j$, for $j=1,2$.
	Meanwhile, we can have the following computations in $-T\leq t\leq T$:
	\begin{align}\label{comp V_j in R^n}
		\begin{split}
			\int_{\R^n} \wt u_j^2 \, dx=& \int_{\R^n} \LC \p_t \widetilde{\mathbf{V}}_j-\mathcal{L}_j \widetilde{\mathbf{V}}_j \RC^2  dx \\
			=& \int_{\R^n} \left[ \LC \p_t\widetilde{\mathbf{V}}_j\RC^2 -2\mathcal{L}_j \widetilde{\mathbf{V}}_j\cdot \p_t \widetilde{\mathbf{V}}_j+\LC\mathcal{L}_j\widetilde{\mathbf{V}}_j \RC^2 \right] dx \\
			=& \int_{\R^n}  \left[ \LC \p_t\widetilde{\mathbf{V}}_j\RC^2 +2\sigma_j\nabla \widetilde{\mathbf{V}}_j \cdot \p_t \nabla \widetilde{\mathbf{V}}_j +\LC\mathcal{L}_j\widetilde{\mathbf{V}}_j \RC^2 \right] dx,
		\end{split}
	\end{align}
	for $j=1,2$, where we have used the integration by parts in the last equality. Moreover, $2\sigma_j\nabla \mathbf{V}_j \cdot \p_t \nabla \widetilde{\mathbf{V}}_j =\frac{d}{dt}\LC \sigma_j \nabla \widetilde{\mathbf{V}}_j \cdot \nabla \widetilde{\mathbf{V}}_j \RC$, and 
	\[
	\int_{-T}^\zeta \int_{\R^n}2\sigma_j\nabla \widetilde{\mathbf{V}}_j \cdot \p_t \nabla \widetilde{\mathbf{V}}_j \, dxdt =\left. \int_{\R^n} \sigma_j \nabla \widetilde{\mathbf{V}}_j \cdot \nabla \widetilde{\mathbf{V}}_j\, dx\right|_{t=-T}^{t=\zeta},
	\]
	for $\zeta \in [-T,T]$ and $j=1,2$. Hence, integrate \eqref{comp V_j in R^n} with respect to the time-variable, then one can show that 
	\begin{align}\label{comp V_j in L^2}
		\begin{split}
			&c_0\max_{-T\leq \zeta\leq T}\int_{\R^n}|\nabla \widetilde{\mathbf{V}}_j|^2 \, dx+\int_{-T}^T\int_{\R^n}\LC\p_t \widetilde{\mathbf{V}}_j\RC^2 \, dxdt +\int_{-T}^T\int_{\R^n}\LC \mathcal{L}_j\widetilde{\mathbf{V}}_j\RC^2 \, dxdt \\
			\leq &\max_{-T\leq \zeta\leq T}\int_{\R^n}\sigma_j\nabla \widetilde{\mathbf{V}}_j\cdot \nabla \widetilde{\mathbf{V}}_j\, dx+\int_{-T}^T\int_{\R^n}\LC\p_t \widetilde{\mathbf{V}}_j\RC^2 \, dxdt +\int_{-T}^T\int_{\R^n}\LC \mathcal{L}_j\widetilde{\mathbf{V}}_j\RC^2 \, dxdt \\
			=&2\int_{-T}^T\int_{\R^n}\wt u_j^2 \, dxdt,
		\end{split}
	\end{align}
	for $j=1,2$, where we used the ellipticity condition \eqref{ellipticity condition} for $\sigma$ and $\wt u_j\in  L^2(\R^{n+1})$. 
	The inequality \eqref{comp V_j in L^2} shows that for a.e. $t\in [-T,T]$, 
	\begin{align}\label{L2 relations}
		\nabla \widetilde{\mathbf{V}}_j(t,\cdot), \quad \p_t \widetilde{\mathbf{V}}_j(t, \cdot )\quad \text{ and }\quad \mathcal{L}_j\widetilde{\mathbf{V}}_j(t,\cdot)\in L^2(\R^n).
	\end{align}

Let us denote $\widetilde{\mathbf{V}}_j':=\frac{\p}{\p t}\widetilde{\mathbf{V}}_j$, multiplying \eqref{wt V_j equation} by $\widetilde{\mathbf{V}}_j'$, then the integration by parts yields that 
	\begin{align*}
		\int_{\R^n} |\widetilde{\mathbf{V}}_j'|^2 \, dx+ \frac{1}{2}\frac{\p}{\p t}\int_{\R^n}\sigma(x)\nabla \widetilde{\mathbf{V}}_j\cdot \nabla \widetilde{\mathbf{V}}_j\, dx =\int_{\R^n}\wt u_j \widetilde{\mathbf{V}}_j'\, dx.
	\end{align*}
Applying the Young's  inequality and integrating with respect to the $t$-variable in the above identity, for any $\varepsilon>0$, we have 
\begin{align*}
&	\int_{-T}^T \norm{\widetilde{\mathbf{V}}_j'}_{L^2(\R^n)}^2\, dt +\frac{1}{2}\sup_{-T\leq t \leq T}\norm{\nabla \widetilde{\mathbf{V}}_j}_{L^2(\R^n)}^2 \\
	\leq &\varepsilon \int_{-T}^T \norm{\widetilde{\mathbf{V}}_j'}_{L^2(\R^n)}^2\, dt +C(\varepsilon)\int_{-T}^T \norm{\wt u_j}_{L^2(\R^n)}^2\, dt.
\end{align*}
In addition, by choosing $\varepsilon>0$ sufficiently small, one can absorb the first term from the right to the left so that 
\begin{align*}
	\int_{-T}^T \norm{\widetilde{\mathbf{V}}_j'}_{L^2(\R^n)}^2\, dt +\sup_{-T\leq t \leq T}\norm{\nabla \widetilde{\mathbf{V}}_j}_{L^2(\R^n)}^2 \leq C\norm{\wt u_j}_{L^2(\R^n)}^2,
\end{align*}
for some constant $C>0$ independent of $\mathbf{V}_j$ and $\wt u_j$, for $j=1,2$.

On the other hand, one can write the equation \eqref{wt V_j equation} in terms of the weak formulation so that 
\begin{align*}
	\int_{\R^n} \widetilde{\mathbf{V}}_j' \phi \, dx + \int_{\R^n}\sigma_j \nabla \widetilde{\mathbf{V}}_j\cdot \nabla \phi \, dx= \int_{\R^n} \wt u_j \phi \, dx,
\end{align*}
for any $\phi=\phi(x)\in H^1(\R^n)$. The above identity is equivalent to 
\begin{align*}
 \int_{\R^n}\sigma \nabla \widetilde{\mathbf{V}}_j\cdot \nabla \phi \, dx= \int_{\R^n} F \phi \, dx,
\end{align*}
for any $\phi=\phi(x)\in H^1(\R^n)$, where $F(t,\cdot):=\wt u_j (t,\cdot )- \widetilde{\mathbf{V}}_j'(t,\cdot)\in L^2(\R^n)$ for a.e. $t\in [-T,T]$. This shows that $\widetilde{\mathbf{V}}_j(t, \cdot)\in H^1(\R^n)$ is a weak solution of $-\nabla \cdot (\sigma_j \nabla \widetilde{\mathbf{V}}_j)=F$ in $\R^n$ for a.e. $t\in [-T,T]$. 
Moreover, the classical interior estimate shows that $\widetilde{\mathbf{V}}_j \in H^2_{\mathrm{loc}}(\R^n)$. In particular, there exists a ball $B_R$ containing $\Omega$, such that $\widetilde{\mathbf{V}}_j  \in H^2(B_R)$. We also observe that 
\begin{align*}
	\Delta\widetilde{\mathbf{V}}_j= \left. \Delta\widetilde{\mathbf{V}}_j \right|_{\Omega}+\left. \Delta\widetilde{\mathbf{V}}_j \right|_{\Omega_e}= \left. \Delta\widetilde{\mathbf{V}}_j \right|_{\Omega}+\left. \nabla \cdot (\sigma \nabla \widetilde{\mathbf{V}}_j) \right|_{\Omega_e} \in L^2 (\R^n),
\end{align*}
where we used $\sigma=\mathbf{I}_n$ in $\Omega_e$ so that the first term in the right hand side of the above identity vanishes in the set $\Omega_e$, and $\widetilde{\mathbf{V}}_j \in H^2(\Omega)\subset H^2(B_R)$. Therefore, due to the fact that $\norm{\Delta \widetilde{\mathbf{V}}_j }_{L^2(\R^n)}^2 =\norm{D^2 \widetilde{\mathbf{V}}_j }_{L^2(\R^n)}^2$, we can show that $\widetilde{\mathbf{V}}_j \in H^2(\R^n)$. In addition, the $H^2$ estimate is independent of $t\in [-T,T]$. Finally, via the definition of $\widetilde{\mathbf{V}}_j$. This proves the assertion.
\end{proof}

\begin{rmk}\label{Rmk: regularity}
	Via Lemma \ref{Lem: regularity}, it is not hard to check that
	\begin{align}\label{V_j regularity}
		\LC \mathcal{H}_j \RC^{s/2} \mathbf{V}_j \in \mathbb{H}^{1,2}(\R^{n+1}),
	\end{align}
	where one can extend $ \LC \mathcal{H}_j \RC^s \mathbf{V}_j $ to be zero for $(t,x)\in \{t> T\} \times \R^n$\footnote{The space $\mathbb{H}^{1,2}(\R^{n+1})$ is introduced in Section \ref{Sec 2}, and the information in the future time domain will not affect the solution in $\Omega_T$.} without loss of generality.
	The preceding lemma will give us desired function spaces for our Cauchy data, which will be utilized in the proof of Theorem \ref{T1}.
\end{rmk}

Let us state the unique continuation principle for the nonlocal parabolic equation \eqref{nonlocal para}, which was shown in \cite[Proposition 5.6]{LLR2019calder} by using suitable Carleman estimate.

\begin{prop}\label{Prop:exterior_UCP}
	Given $s\in (0,1)$, $n\in \N$ and arbitrarily nonempty open sets $W_1, W_2\subset \Omega_e $. Let $u_j \in \mathbb{H}^s(\R^{n+1})$ with $\mathrm{supp}(u_j)\subset \LC \overline{\Omega}\cup \overline{W_1}\RC_T$, for $j=1,2$. Suppose that 
	\begin{align}
		\label{eq:overdet}
		u_1=u_2 \in C^\infty_c\LC\LC W_1\RC _T\RC  \quad  \text{ and }  \quad  \LC \mathcal{H}_1\RC^s u_1 =\LC \mathcal{H}_2\RC^s u_2  \text{ in } \LC W_2\RC_T.
	\end{align} 
	Then $\mathbf{U}_1=\mathbf{U}_2$ in $ (-T,T)\times (0,\infty)\times \R^n$, where $\mathbf{U}_j$ is defined by \eqref{conv-heat-def}, for $j=1,2$.
\end{prop}

Recalling the nonlocal parabolic operator $\mathcal{H}^s$ is defined via \eqref{H^s}, with the condition \eqref{eq:overdet} at hand, one has that 
\begin{align}\label{zero cond of U_j}
\begin{split}
\int_0^\infty  \frac{ \mathbf{U}_1(t,\tau,x) -\mathbf{U}_2(t,\tau,x)}{\tau^{1+s}}\, d\tau=0, \quad \text{ for }(t,x)\in \LC W_2\RC_T, 
\end{split}
\end{align}
where $\mathbf{U}_j$ is given by \eqref{conv-heat-def}, for $j=1,2$. By utilizing the condition \eqref{zero cond of U_j}, we can prove the proposition.

\begin{proof}[Proof of Proposition \ref{Prop:exterior_UCP}]
	Inspired by the proof of \cite[Proposition 3.1]{GU2021calder}, let us consider bounded open set $\mathcal{O}_j \Subset W_j\subset \Omega_e$ ($j=1,2$) such that $\overline{\mathcal{O}_1}\cap \overline{\mathcal{O}_2}=\emptyset$. Without loss of generality, we may assume that $\mathrm{supp}(u_j)\subset \LC \overline{\Omega}\cup \overline{\mathcal{O}_1}\RC _T$, for $j=1,2$. Consider 
	$$
	\mathbf{U}:=\mathbf{U}_1-\mathbf{U}_2,
	$$ 
	then one can have 
	\begin{align}
		\begin{split}
			\mathbf{U}(t,\tau,x) = &\mathbf{U}_1(t,\tau,x) - \mathbf{U}_2(t,\tau,x)  \\
			=& \int_{\Omega \cup \mathcal{O}_1} p_1(x,z,\tau)u_1(t-\tau,z)\, dz-\int_{\Omega \cup \mathcal{O}_1} p_2(x,z,\tau)u_2(t-\tau,z)\, dz,
		\end{split}
	\end{align}
where we have utilized the condition $\mathrm{supp}(u_j)\subset \LC\overline{\Omega} \cup \overline{\mathcal{O}_1}\RC_T$ and $p_j(x,z,\tau)$ is the corresponding heat kernel of $\p_\tau+\mathcal{L}_j$, for $j=1,2$. Moreover, it is known that heat kernels $p_j(x,z,\tau)$ satisfies 
\begin{align}\label{est-heat-kernel}
	C_1 \LC \frac{1}{4\pi \tau}\RC^{n/2}e^{-\frac{c_1|x-z|^2}{4\tau}} \leq p_j(x,z,\tau)\leq C_2 \LC \frac{1}{4\pi \tau}\RC^{n/2}e^{-\frac{c_2|x-z|^2}{4\tau}},
\end{align}
for $j=1,2$, for some positive constants $c_1,c_2,C_1$ and $C_2$.

\medskip

{\it Claim 1. $\dfrac{\mathbf{U}(t,\tau,x)}{\tau^{N+s}}\in L^1(0,\infty)$, for all $N\in \N$, and for any given $(t,x)\in \LC \mathcal{O}_2 \RC_T$. }

\medskip

\noindent  In order to show the claim, one can examine whether the integral $\int_0^\infty \left| \frac{\mathbf{U}(t,\tau,x)}{\tau^{N+s}}\right|\,  d\tau$ is finite or not. Similar to the arguments as in the proof of \cite[Proposition 3.1]{GU2021calder}, given $\delta \in (0,1)$, we can divide the integral 
\begin{align*}
	\int_0^\infty \left| \frac{\mathbf{U}(t,\tau,x)}{\tau^{N+s}}\right|\,  d\tau =I_\delta + II_\delta,
\end{align*}
where 
\begin{align*}
	I_\delta:=\int_0^\delta \left| \frac{\mathbf{U}(t,\tau ,x)}{\tau^{N+s}}\right|\,  d\tau \quad \text{ and }\quad 	 II_\delta:=\int_\delta^\infty \left| \frac{\mathbf{U}(t,\tau,x)}{\tau^{N+s}}\right|\,  d\tau. 
\end{align*}
For $II_\delta$, by using the H\"older's inequality, one can see that 
\begin{align}\label{II_delta}
	\begin{split}
		II_\delta \leq C \LC \norm{u_1}_{L^2(\R^{n+1})}+\norm{u_2}_{L^2(\R^{n+1})}\RC\LC \int_{\delta}^\infty \frac{1}{\tau^{2N+2s}}\, d\tau \RC^{\frac{1}{2}}<\infty,
	\end{split}
\end{align}
for some constant $C>0$ independent of $\tau>0$.
On the other hand, for $I_\delta$, using the H\"older's inequality and the property of the heat kernel estimate \eqref{est-heat-kernel}, we have that 
\begin{align}\label{I_delta}
\begin{split}
		I_{\delta} \leq & C \LC \norm{u_1}_{L^2(\R^{n+1})}+\norm{u_2}_{L^2(\R^{n+1})}\RC \LC \int_0^\delta \int_{\Omega \cup \mathcal{O}_1} \frac{e^{-\frac{|x-z|^2}{\tau}}}{\tau^{2N+2s}}\, dz d\tau \RC^{1/2} \\
		\leq & \wt C \LC \int_0^\delta \frac{e^{-\frac{\kappa^2}{\tau}}}{\tau^{2N+2s}}\, d\tau \RC^{1/2}<\infty , 
\end{split}
\end{align}
for some constants $C, \wt C>0$. Here we have used that $\Omega$ and $\mathcal{O}_1$ are bounded sets in $\R^n$, and $x\in \mathcal{O}_2$, such that $|x-z|\geq \kappa>0$,  for some $\kappa>0$ (recalling that $z\in \Omega \cup \mathcal{O}_2$ and $\overline{\Omega \cup \mathcal{O}_2}\cap \overline{\mathcal{O}_1}=\emptyset$). Combining with \eqref{II_delta} and \eqref{I_delta}, one can conclude that $\frac{\mathbf{U}(t,\tau,x)}{\tau^{N+s}}\in L^1(0,\infty)$ for all $N\in \N$, and for any given $(x,t)\in \LC \mathcal{O}_2 \RC_T$.

\medskip

{\it Claim 2. $\displaystyle\int_0^\infty \frac{\mathbf{U}(t,\tau ,x)}{\tau^{N+s}}\, d\tau =0$, for all $N\in \N$, and for any given $(t,x)\in \LC \mathcal{O}_2 \RC_T$.}

\medskip

\noindent With the equation \eqref{U_j equation} at hand, notice that 
$$\left. \mathcal{H}_1\right|_{(\Omega_e)_T}=\left. \mathcal{H}_2\right|_{(\Omega_e)_T}=\left. \LC \p_t -\Delta\RC \right|_{(\Omega_e)_T}:=\mathcal{H}|_{(\Omega_e)_T}, 
$$ 
then one can see that $\mathbf{U}$ is a solution of 
\begin{align}\label{exterior heat}
	\begin{cases}
		\p_\tau \mathbf{U}=-\mathcal{H}\mathbf{U}  &\text{ in }(-T,T)\times (0,\infty) \times \Omega_e, \\
		\mathbf{U}(t,0,x)=0 & \text{ for }(t,x) \in \LC\Omega_e\RC_T,
	\end{cases}
\end{align}
where we utilized the condition that $\mathbf{U}(t,0,x)=\mathbf{U}_1(t,0,x)-\mathbf{U}_2(t,0,x)=u_1(t,x)-u_2(t,x)=0$ in $\LC \Omega_e \RC_T$. Via the condition \eqref{zero cond of U_j}, the function $\mathbf{U}$ satisfies 
\begin{align}\label{N=1}
\int_0^\infty \frac{\mathbf{U}(t,\tau,x)}{\tau^{1+s}}\, d\tau =0 \text{ in }  \LC\mathcal{O}_2\RC_T,
\end{align}
which proves the {\it Claim 2} for the case $N=1$.

Furthermore, since $u_j(t,x)$ are $C^\infty$-smooth for $(t,x)\in \LC \Omega_e\RC_T$, and $p_j(x,z,\tau)$ is also smooth, for $j=1,2$, then we get $\mathbf{U}( t ,\tau,x)$ is smooth in the $(t,x)$-variables, for $(t,x)\in \LC\Omega_e\RC_T$. Hence, by applying the heat operator $\mathcal{H}^m$ to the equation \eqref{exterior heat} for any $m\in \N \cup \{0\}$, one has that 
\begin{align}\label{exterior heat-high}
	\begin{cases}
		\LC  \p_\tau +\mathcal{H}\RC \mathcal{H}^m\mathbf{U}=0 &\text{ in }(-T,T) \times (0,\infty) \times \Omega_e,\\
		\mathcal{H}^m\mathbf{U}(t,0,x)=0 &\text{ in }\LC \Omega_e\RC_T.
	\end{cases}
\end{align}
Meanwhile, similar to the arguments as in the {\it Claim 1}, we can show that 
\begin{align*}
	\frac{\mathcal{H}^m\mathbf{U}(t,\tau,x)}{\tau^{1+s}} \in L^1(0,\infty), \text{ for }(t,x)\in  \LC \mathcal{O}_2\RC_T, \text{ and for any }m\in \N\cup\{0\}.
\end{align*}
For the case $m=N+1$, $N\in \N$, by acting $\mathcal{H}^{N+1}$ on \eqref{N=1}, we obtain that 
\begin{align*}
	\int_0^\infty \frac{\mathcal{H}\LC\mathcal{H}^{N}\mathbf{U}(t,\tau,x)\RC}{\tau^{1+s}}\, d\tau=\mathcal{H}^{N+1}\LC \int_0^\infty \frac{\mathbf{U}(t,\tau,x)}{\tau^{1+s}}\, d\tau \RC =0 \text{ in }  \LC\mathcal{O}_2\RC_T,
\end{align*}
which is equivalent to 
\begin{align}\label{m=N+1 tau}
	\int_0^\infty \frac{\p_\tau \LC \mathcal{H}^{N}\mathbf{U}(t,\tau,x)\RC}{\tau^{1+s}}\, d\tau=0 \text{ in }  \LC\mathcal{O}_2\RC_T,
\end{align}
where we used the equation \eqref{exterior heat-high}.

Now, an integration by parts for \eqref{m=N+1 tau} yields that 
\begin{align}\label{m=N+1 tau int by p}
	\begin{split}
		0=&\left[\frac{  \mathcal{H}^{N}\mathbf{U}(t,\tau,x)}{\tau^{1+s}}\right]_{\tau=0}^{\tau=\infty}-\int_0^\infty  \mathcal{H}^{N}\mathbf{U}(t,\tau ,x) \p_\tau\LC \frac{ 1}{\tau^{1+s}}\RC\, d\tau\\
		=&(1+s)\int_0^\infty \frac{ \mathcal{H}^{N}\mathbf{U}(t,\tau,x)}{\tau^{2+s}}\, d\tau \quad \text{ in }  \LC\mathcal{O}_2\RC_T,
	\end{split}
\end{align}
where we used that $\left[\frac{  \mathcal{H}^{N}\mathbf{U}(t,\tau,x)}{\tau^{1+s}}\right]_{\tau=0}^{\tau=\infty}=0$.
As a result, by repeating the preceding arguments for $m=N-1,N-2,\ldots, 1$, with \eqref{m=N+1 tau int by p} at hand, we can conclude that 
\begin{align}\label{m=N+1 final}
	\int_0^\infty \frac{\mathbf{U}(t,\tau,x)}{\tau^{N+s}}\, d\tau=0 \text{ in }  \LC\mathcal{O}_2\RC_T,
\end{align}
and this proves the claim.

With \eqref{m=N+1 final} at hand, for any $\xi \in \R$, since $\frac{\mathbf{U}(t,\tau,x)}{\tau^{1+s}}\in L^1(0,\infty)$ for $(t,x)\in   \LC \mathcal{O}_2 \RC_T$, then $\int_0^\infty \frac{\mathbf{U}(t,\tau,x)}{\tau^{1+s}}e^{\frac{\mathbf{i}\xi}{\tau}}\, d\tau$, for $(t,x)\in \LC \mathcal{O}_2 \RC_T$ exists. Moreover, by using \eqref{m=N+1 final} again, one can  obtain that 
\begin{align}\label{vanish int}
	\int_0^\infty \frac{\mathbf{U}(t,\tau,x)}{\tau^{1+s}}e^{\frac{\mathbf{i}\xi}{\tau}}\, d\tau =	\int_0^\infty \frac{\mathbf{U}(t,\tau,x)}{\tau^{1+s}} \LC \sum_{k=0}^\infty \frac{1}{k!}\frac{(\mathbf{i}\xi)^k}{\tau^k} \RC d\tau=0,
\end{align}
for any $\xi \in \R$ and $(t,x)\in \LC \mathcal{O}_2\RC_T$.
Hence, by using the change of variable $\tau=\alpha^{-1}$, the integral \eqref{vanish int} is equivalent to 
\begin{align}\label{vanish int 1}
 \int_{0}^\infty \frac{\mathbf{U}(t,\alpha^{-1},x)}{\alpha^{1-s}}e^{\mathbf{i}\xi \alpha}\, d\alpha=0, \quad \text{ for all }\xi \in \R,
\end{align}
which can be regarded as the one-dimensional Fourier transform with respect to the $\alpha$-variable (here we can extend the function $\mathbf{U}(t,\alpha^{-1},x)=0$ for $\alpha<0$). Therefore, \eqref{vanish int 1} implies that 
\begin{align}\label{zero U in smaller domain}
	\mathbf{U}(t,\tau,x)=0, \text{ for } (t,\tau,x)\in \LC -T,T\RC\times (0,\infty) \times \mathcal{O}_2
\end{align}
as we wish.
Finally, by using the (weak) unique continuation of \eqref{exterior heat} (Section \ref{Sec: UCP}), we can show that $\mathbf{U}=0$ in $\LC\Omega_e\RC_T\times (0,\infty)$, which is equivalent to 
\begin{align}\label{bf U_1=U_2}
	\mathbf{U}_1=\mathbf{U}_2 \text{ in }\LC -T,T\RC\times (0,\infty) \times \Omega_e.
\end{align} 
This proves the assertion.	
\end{proof}

Moreover, we can show the global unique continuation property for $\mathcal{H}^s$.

\begin{lem}[Global unique continuation property]\label{Prop:global_UCP}
Let $u\in \mathbb{H}^s(\R^{n+1})$, and suppose that $u=\mathcal{H}^su=0$ in $\mathcal{O}_T$, where $\mathcal{O}\subset \R^n$ is an arbitrarily open set. Then we have $u\equiv  0$ in $(\R^n)_T$. 
\end{lem}

\begin{proof}
	The proof has been demonstrated by \cite[Theorem 1.3]{BKS2022calderon}, whenever the leading coefficient $\sigma$ is globally Lipschitz continuous on $\R^n$. The arguments are based on suitable unique continuation properties for degenerate parabolic equations (Proposition \ref{Prop: extension}), and we refer readers to the detailed explanation in \cite{BKS2022calderon}.
\end{proof}

\begin{rmk}\label{Rmk: AdjUCP}
	With the preceding lemma at hand, it is not hard to see that the global unique continuation property also holds for the adjoint parabolic operator $\mathcal{H}^s_\ast$. In other words, given a nonempty open set $\mathcal{O}\subset \R^n$, if $v=\mathcal{H}^s_\ast v=0$ in $\mathcal{O}_T$, then $v\equiv 0$ in $\LC \R^n\RC_T$ as well. The proof can be achieved by repeating the arguments from Proposition \ref{Prop:exterior_UCP}, where one replaces the parabolic operator $\mathcal{H}=\p_t +\mathcal{L}$ by $\mathcal{H}_\ast=-\p_t +\mathcal{L}$.
\end{rmk}

\subsection{Proof of Theorem \ref{T1}}

We divide the proof of Theorem \ref{T1} into two parts.
\begin{proof}[Proof of Theorem \ref{T1}--Part 1]

Recalling that $\mathbf{V}_j=\mathbf{V}_j(t,x)$ is the function defined by \eqref{V_j}, via Proposition \ref{Prop:exterior_UCP}, one has 
\begin{align}
	\mathbf{V}_1 =\int_0^\infty \mathbf{U}_1(t,\tau ,x) \, d\tau =\int_0 ^\infty \mathbf{U}_2(t,\tau ,x)\, d\tau =\mathbf{V}_2 \text{ in }(\Omega_e)_T,
\end{align}
and $\mathbf{V}_j$ satisfies \eqref{V_j equation}, for $j=1,2$.

Define the function
\begin{align}\label{W_j}
	\mathbf{W}_j := \LC \mathcal{H}_j \RC^s \mathbf{V}_j \text{ in }\LC \R^n\RC_T, \text{ for }j=1,2.
\end{align}
We observe that 
\begin{align}\label{evolutive V_j=0}
	\mathcal{P}^{\mathcal{H}_j}_\tau \mathbf{V}_j(-T,x)=\int_{\R^n}p_j (x,z,\tau ) \mathbf{V}_j (-T-\tau,x)\, dz=0,
\end{align}
for $j=1,2$ and for all $\tau \in (0,\infty)$, where we have utilized that \eqref{past time V_j}.
Combining with the definition of $\LC \mathcal{H}_j\RC^s$, \eqref{V_j equation} and \eqref{evolutive V_j=0}, one has that 
\begin{align}\label{W_j initial}
\begin{split}
		\mathbf{W}_j(-T,x)= & \LC \LC \mathcal{H}_j \RC^s \mathbf{V}_j\RC(-T,x) \\
	=& -\frac{s}{\Gamma(1-s)}\int_0^\infty \frac{\mathcal{P}^{\mathcal{H}_j}_\tau \mathbf{V}_j(-T,x)-\mathbf{V}_j(-T,x)}{\tau^{1+s}}\, d\tau \\
	=&0, 
\end{split}
\end{align}
for $x \in \R^n$. Moreover, by interchanging the local and nonlocal parabolic operators, one has that 
\begin{align}\label{interchanging}
	\mathcal{H}_j  \LC  \LC \mathcal{H}_j \RC^s  \mathbf{V}_j \RC = \LC \mathcal{H}_j \RC^s \LC  \mathcal{H}_j \mathbf{V}_j\RC \text{ in }\LC \R^n\RC_T.
\end{align}
Acting $\mathcal{H}_j$ on \eqref{V_j equation}, by using \eqref{W_j initial} and  \eqref{interchanging}, one obtains that
\begin{align}\label{W_j equation}
	\begin{cases}
		\mathcal{H}_j \mathbf{W}_j = \LC  \mathcal{H}_j\RC^s u_j &\text{ in }\LC \R^n\RC_T,\\
		\mathbf{W}_j(-T,x)=0 & \text{ in } \R^n.
	\end{cases}
\end{align}
By applying the condition \eqref{V_j regularity}, we have 
\begin{align}\label{W_regularity}
	\mathbf{W}_j\in \mathbb{H}^{1-s/2,2-s}(\R^{n+1})\, (=\mathbb{H}^{2-s}(\R^{n+1})),
\end{align}
so that
\begin{align}\label{W_regularity 2}
	\mathbf{W}_j\in L^2(0,T;H^1(\R^n)) 
\end{align} 
due to $s \in (0,1)$, for $j=1,2$.
On the other hand, recalling that $u_j$ satisfies \eqref{nonlocal para sec 3}, we have 
\begin{align*}
	\mathcal{H}_j \mathbf{W}_j = 0 \text{ in } \Omega_T, \text{ for }j=1,2.
\end{align*}

We next claim that 
\begin{align}\label{claim W1=W2}
	\mathbf{W}_1=\mathbf{W}_2 \text{ in }(\Omega_e)_T.
\end{align}
In order to show \eqref{claim W1=W2}, we consider another function 
\begin{align}\label{Ubb_j definition}
	\mathbb{U}_j(t,\tau,x):= \int_{\R^n} p_j(x,z,\tau) \mathbf{V}_j(t-\tau,z)\, dz,
\end{align}
as in Lemma \ref{Lem: Key}, $\mathbb{U}_j$ solves 
	\begin{align}\label{Ubb_j equation}
	\begin{cases}
		\p_\tau  \mathbb{U}_j(t,\tau,x)+\mathcal{H}_j \mathbb{U}_j(t,\tau,x)=0, & \text{ for }(t,\tau,x)\in \R \times (0,\infty) \times \R^n, \\
		\mathbb{U}_j(t,0,x)=\mathbf{V}_j(t,x) &\text{ for }(t,x)\in \R^{n+1},
	\end{cases}
\end{align}
for $j=1,2$. Now, by acting the parabolic operator $\mathcal{H}_j$ on both sides of \eqref{Ubb_j equation}, we can get
\begin{align}\label{tilde Ubb_j equation}
	\begin{cases}
		\p_\tau \widetilde{\mathbb{U}}_j + \mathcal{H}_j \widetilde{\mathbb{U}}_j=0  & \text{ for }(x,t,\tau)\in \R^{n+1}\times (0,\infty), \\
		\widetilde{\mathbb{U}}_j (t,0,x)=u_j(t,x) &\text{ for }(t,x)\in \R^{n+1},
	\end{cases}
\end{align}
where $\widetilde{\mathbb{U}}_j:=\mathcal{H}_j \mathbb{U}_j$ and we used the equations \eqref{V_j equation} and \eqref{Ubb_j equation} in the second equality of \eqref{tilde Ubb_j equation}, for $j=1,2$. More precisely, from \eqref{V_j equation} and \eqref{Ubb_j equation}, we have that 
\begin{align*}
	\begin{split}
		\widetilde{\mathbb{U}}_j(t,0,x)=\mathcal{H}_j \mathbb{U}_j (t,0,x)=\mathcal{H}_j\mathbf{V}_j(t,x)=u_j(t,x),
	\end{split}
\end{align*}
for $j=1,2$.
Furthermore, by \eqref{Ubb_j definition}, it is known that 
\begin{align}\label{Ubb_j=0}
	\begin{split}
		\mathbb{U}_j(-T,\tau,x)= \int_{\R^n} p_j(x,z,\tau) \mathbf{V}_j(-T-\tau,z)\, dz=0, 
	\end{split}
\end{align}
for all $\tau \in (0,\infty)$, where we used the condition \eqref{past time V_j}. Via the definition of $\widetilde{\mathbb{U}}_j$, \eqref{tilde Ubb_j equation}, and \eqref{Ubb_j=0}, one has that 
\begin{align}\label{tilde U_bb -T=0}
 \begin{split}
 	\widetilde{\mathbb{U}}_j(-T,\tau,x)=\mathcal{H}_j \mathbb{U}_j(-T,\tau,x)
 	=-\p_\tau \mathbb{U}_j(-T,\tau,x)=0.
 \end{split}
\end{align}

By Corollary \ref{Cor: unique solution}, combining with the condition \eqref{tilde U_bb -T=0}, the equation \eqref{tilde Ubb_j equation} possesses a unique solution. Now, comparing the equations \eqref{U_j equation} and \eqref{tilde Ubb_j equation}, they both have the same initial condition 
$$\mathbf{U}_j(t,0,x)=u_j(x,t)=\widetilde{\mathbb{U}}_j(t,0,x) \text{ for } (t,x)\in \R^{n+1},
$$ 
which yields that 
\begin{align}\label{unique of several U_j}
	\mathbf{U}_j(t,\tau,x)=\widetilde{\mathbb{U}}_j(t,\tau,x)=\mathcal{H}_j \mathbb{U}_j, \text{ for } (t,\tau,x)\in \R\times  (0,\infty)\times \R^{n},
\end{align}
for $j=1,2$.
Thus, by using \eqref{Ubb_j equation} and \eqref{unique of several U_j}, we have 
\begin{align}\label{id 1}
	\p_\tau \mathbb{U}_1 - \p_\tau \mathbb{U}_2 = -\mathcal{H}_1 \mathbb{U}_1 +\mathcal{H}_2 \mathbb{U}_2=-\mathbf{U}_1 +\mathbf{U}_2 =0 \text{ in }\LC -T,T\RC\times (0,\infty) \times \Omega_e,
\end{align}
where the last equality holds due to the identity \eqref{bf U_1=U_2}.

In addition, thanks to the identity \eqref{id 1}, we know that 
\[
\LC \mathbb{U}_1 -\mathbb{U}_2\RC(t,\tau,x)=\LC \mathbb{U}_1 -\mathbb{U}_2\RC(t,0,x),  \text{ for }(t,\tau,x)\in \LC -T,T\RC\times (0,\infty)\times \Omega_e,
\]
which is equivalent to 
\begin{align}\label{id 2}
\mathbb{U}_1  (t,\tau,x)- \mathbb{U}_1 (t,0,x)= \mathbb{U}_2(t,\tau ,x)-\mathbb{U}_2(t,0,x) ,  
\end{align}
for $(t,\tau,x)\in \LC-T, T\RC\times (0,\infty) \times  \Omega_e$.
Consequently, \eqref{id 2} implies that 
\begin{align*}
	\int_0^\infty \frac{\mathbb{U}_1(t,\tau,x)-\mathbb{U}_1 (t,0,x)}{\tau^{1+s}}\, d\tau =	\int_0^\infty \frac{\mathbb{U}_2(t,\tau,x)-\mathbb{U}_2(t,0,x)}{\tau^{1+s}}\, d\tau, 
\end{align*}
for $(t,x)\in \LC\Omega_e\RC_T$. Meanwhile, via the definition \eqref{H^s} of nonlocal parabolic operators, the above identity gives rise to 
\begin{align}\label{id 3}
	\LC\mathcal{H}_1\RC^s \mathbf{V}_1=\LC \mathcal{H}_2\RC^s \mathbf{V}_2 \text{ in }\LC \Omega_e \RC_T.
\end{align}
Recall that the function $\mathbf{W}_j$ is defined by \eqref{W_j} for $j=1,2$, then \eqref{id 3} infers that the claim \eqref{claim W1=W2} holds.

Hence, combining with \eqref{W_regularity 2}, $\mathbf{W}_j\in L^2(-T,T ;H^1(\R^n))$ satisfies 
\begin{align*}
	\begin{cases}
		\mathcal{H}_j \mathbf{W}_j=0 &\text{ in }\Omega_T, \\
		\left\{  \mathbf{W}_1 , \, \sigma_1\p_{\nu} \mathbf{W}_1 \right\}=\left\{ \mathbf{W}_2 , \, \sigma_2\p_{\nu} \mathbf{W}_2  \right\}&\text{ on }\Sigma_T,
	\end{cases}
\end{align*}
where $\sigma_j\p_\nu \mathbf{W}_j$ denotes the Neumann data on $\Sigma_T$ given by \eqref{conormal}, for $j=1,2$. Moreover, by the trace theorem, it is known that 
$$
\left\{  \mathbf{W}_1 , \, \sigma_1\p_{\nu} \mathbf{W}_1 \right\} \in L^2(0,T;H^{1/2}(\Sigma)) \times L^2(0,T;H^{-1/2}(\Sigma)). 
$$
Finally, it remains to show that whether we can vary all possible Dirichlet data so that we are able to reduce nonlocal inverse problems to local ones, and the rest of the arguments will be given in next section.
\end{proof}


We want to show that the lateral boundary Cauchy data 
$$
\mathcal{C}^{(j)}_{\Sigma_T}=\left\{ \left. \mathbf{V}_j \right|_{\Sigma_T}, \, \left. \sigma_j \p_{\nu} \mathbf{V}_j \right|_{\Sigma_T} \right\},
$$ 
where $\mathbf{V}_j$ is a solution of the initial-boundary value problem
\begin{align*}
	\begin{cases}
		\mathcal{H}_j\mathbf{V}_j=0 &\text{ in }\Omega_T,\\
		\mathbf{V}_j=f &\text{ on }\Sigma_T, \\
		\mathbf{V}_j(-T,x)=0 &\text{ for }x\in \Omega,
	\end{cases}
\end{align*}
for $j=1,2$. Our aim is to prove 
\begin{align}\label{same local Cauchy}
	\mathcal{C}^{(1)}_{\Sigma_T}=\mathcal{C}^{(2)}_{\Sigma_T}.
\end{align}

We first demonstrate a connection between local and nonlocal Calder\'on problems. Adopting all notations in previous sections, we further define two solution spaces that 
\begin{align*}
	\mathcal{D}_j(\Omega_T):= \left\{ \left.\mathbf{V}_j\right|_{\Omega_T}: \, \begin{cases}
		\mathcal{H}_j \mathbf{V}_j=0 &\text{ in }\Omega_T,\\
		\mathbf{V}_j(-T,x)=0&\text{ for }x\in \Omega,
	\end{cases} \right\}
\end{align*}
and
\begin{align}\label{E_j}
 \mathcal{E}_j(\Omega_T):=\left\{ \left.\mathbf{W}_j\right|_{\Omega_T}: \,  \begin{cases}
 	\mathcal{H}_j \mathbf{W}_j=\LC \mathcal{H}_j\RC^s u_j &\text{ in } \R^n_T \\
 	\mathbf{W}_j(-T,x)=0 &\text{ for }x\in \R^n
 \end{cases}  \right\},
\end{align}
where $u_j\in \mathbb{H}^s(\R^{n+1})$ is the solution of \eqref{nonlocal para sec 3}, for $j=1,2$.
Then we are able to show:

\begin{lem}\label{Lem: Density}
		$\mathcal{E}_j(\Omega_T)$ is dense in $\mathcal{D}_j(\Omega_T)$ with respect to $L^2(-T,T;H^1(\Omega))$, for $j=1,2$.
\end{lem}

We first assume that Lemma \ref{Lem: Density} holds, then we can complete the proof of Theorem \ref{T1}. 

\begin{proof}[Proof of Theorem \ref{T1}--Part 2]
Given any $\mathbf{V}_j\in \mathcal{D}_j$ with $\mathbf{V}_1=\mathbf{V}_2=f$ on $\Sigma_T$ for arbitrary $f\in L^2(-T,T;H^{1/2}(\Sigma))$, then there must exist sequences $\left\{ \mathbf{W}_j^{(k)}\right\}_{k\in \N}$ solves \eqref{W_j equation} such that $\mathbf{W}_j^{(k)} \rightarrow \mathbf{V}_j$ in $L^2(-T,T;H^1(\Omega))$ as $k \to \infty$, for $j=1,2$. Similar as the Part 1 of the proof of Theorem \ref{T1}, $\mathbf{W}_j^{(k)}$ satisfies 
\begin{align}\label{Wk sequence}
	\begin{cases}
		\mathcal{H}_j \mathbf{W}^{(k)}_j=0 &\text{ in }\Omega_T, \\
		\left\{  \mathbf{W}^{(k)}_1 , \, \sigma_1\p_{\nu} \mathbf{W}^{(k)}_1 \right\}=\left\{ \mathbf{W}^{(k)}_2 , \, \sigma_2\p_{\nu} \mathbf{W}^{(k)}_2  \right\}&\text{ on }\Sigma_T,
	\end{cases}
\end{align}
for $k\in \N$ and $j=1,2$.
By taking the limit $k\to \infty$ of \eqref{Wk sequence}, we can have 
\begin{align}\label{limit U_j}
	\begin{cases}
		\mathcal{H}_j \mathbf{V}_j=0 &\text{ in }\Omega_T, \\
		\left\{  \mathbf{V}_1 , \, \sigma_1\p_{\nu} \mathbf{V}_1 \right\}=\left\{ \mathbf{V}_2 , \, \sigma_2\p_{\nu} \mathbf{V}_2  \right\}&\text{ on }\Sigma_T,
	\end{cases}
\end{align}
for $j=1,2$. Hence, we show that \eqref{same local Cauchy} holds. This shows Theorem \ref{T1} holds true.
\end{proof}

\begin{prop}\label{Prop: dense 1}
	Let $\Omega\subset \R^n$, $0<s<1$, and $u\in \mathbb{H}^s(\R^{n+1})$ satisfy 
	\begin{align}\label{sol u pr}
		\mathcal{H}^s u =0 \text{ in }\Omega_T.
	\end{align}
Then for any open set $\mathcal{O}\subset \R^n \setminus \overline{(\Omega \cup \mathcal{O}_1)}$, the set 
\begin{align*}
	\mathcal{X}((\Omega_e)_T):=\left\{ \left. \mathcal{H}^su \right|_{(\Omega_e)_T}: \,  u \text{ is a solution to \eqref{sol u pr}} \right\}
\end{align*}
is dense in $\mathbf{H}^{-s}((\Omega_e)_T)$.
\end{prop}


\begin{proof}
	By the Hahn-Banach theorem, it suffices to show that given $\varphi \in \widetilde{\mathbf{H}}^s((\Omega_e)_T)$ such that 
	\begin{align}\label{pari =0}
	\left\langle  \mathcal{H}^s u, \varphi  \right\rangle_{(\Omega_e)_T}\equiv \left\langle  \mathcal{H}^s u, \varphi \right\rangle_{\mathbf{H}^{-s}((\Omega_e)_T)\times \widetilde{\mathbf{H}}^s((\Omega_e)_T)}=0,
	\end{align}
for any solutions $u$ of \eqref{sol u pr}, then we want to claim $\varphi \equiv0$.

Consider the adjoint problem and let $v\in \mathbb{H}^s(\R^{n+1})$ be the solution of 
\begin{align}\label{adjoint prob}
   \begin{cases}
   	\mathcal{H}^s_\ast v =0 &\text{ in }\Omega_T,\\
   	v=\varphi &\text{ in }(\Omega_e)_T.
   \end{cases}
\end{align}
Now, via equations \eqref{sol u pr}, \eqref{adjoint prob} and \eqref{pari =0}, one has that 
\begin{align}\label{pair =0}
 \begin{split}
 		\left\langle u,\mathcal{H}^s_\ast v \right\rangle_{(\Omega_e)_T}
 	=&\left\langle  u ,\mathcal{H}^s_\ast v \right\rangle_{\LC \R^n\RC_T}-\left\langle u,\mathcal{H}^s_\ast v \right\rangle_{\Omega_T}\\=
 	& \left\langle  \mathcal{H}^s u, v \right\rangle_{\LC \R^n\RC_T} \\
 	=& \left\langle  \mathcal{H}^s u, \varphi \right\rangle_{(\Omega_e)_T}=0.
 \end{split}
\end{align}
Thus, since $u|_{(\Omega_e)_T}$ can be arbitrary, by varying the value $u|_{(\Omega_e)_T}\in C^\infty_c\LC(\Omega_e)_T \RC$ and combining with \eqref{pair =0}, one can conclude that $\mathcal{H}_\ast^s v =0 $ in $(\Omega_e)_T$. 
Hence, $v=\mathcal{H}^s_\ast v=0$ in $(\Omega_e)_T$, by applying Remark \ref{Rmk: AdjUCP}, one obtains $v\equiv 0$ in $\LC\R^n\RC_T$. By using the equation \eqref{adjoint prob}, we have $\varphi=v=0$ in $(\Omega_e)_T$ as desired. This proves the assertion.
\end{proof}

We are ready to show Lemma \ref{Lem: Density}.

\begin{proof}[Proof of Lemma \ref{Lem: Density}]
	 Consider $F\in \LC L^2(-T,T;H^{1}(\Omega))\RC^\ast$, which denotes the dual space of $L^2(-T,T;H^1(\Omega))$. 
	 Moreover, by using the definition of dual spaces via the natural dual pairing, it is not hard to see that 
	 \[
	  \LC L^2(-T,T;H^{1}(\Omega))\RC^\ast= L^2(-T,T;\wt H^{-1}(\Omega)),
	 \]
	 where 
	 \[
	 \wt H^{-1}(\Omega):=\left\{ h\in H^{-1}(\R^n):\,  \mathrm{supp}(h)\subset \overline{\Omega}\right\}
	 \]
	 denotes the dual space of $H^1(\Omega)$. In further, we also denote $H^{-1}(\R^n)$ as the dual space of $H^1(\R^n)$.
	 By the Hahn-Banach theorem, it is equivalent to show that 
	 \begin{align}\label{F W=0 1}
	    \left\langle  F, \mathbf{W}_j \right\rangle_{L^2(-T,T;\wt H^{-1}(\Omega)) \times L^2(-T,T;H^1(\Omega))}=0, \quad \text{ for all }\mathbf{W}_j \in \mathcal{E}_j,
	 \end{align}
     then it follows 
      \begin{align}\label{F W=0 2}
     	\left\langle  F,  \mathbf{V}_j \right\rangle_{L^2(-T,T;\wt H^{-1}(\Omega)) \times L^2(-T,T;H^1(\Omega))}=0\quad \text{ for all }\mathbf{V}_j \in \mathcal{D}_j,
     \end{align}
	for $j=1,2$.

For $0<s<1$, recalling that $\mathbf{W}_j$ is the solution of \eqref{W_j equation} for $j=1,2$. By varying the exterior data $f|_{(\Omega_e)_T}\in C^\infty_c \LC (\Omega_e)_T \RC$, Proposition \ref{Prop: dense 1} implies that the set 
\begin{align}\label{Y(O)}
	\mathcal{Y}((\Omega_e)_T):=\left\{ \left. \mathcal{H}_j \mathbf{W}_j \right|_{(\Omega_e)_T}: \,  \mathbf{W}_j \text{ is a solution of \eqref{W_j equation}}  \right\}
\end{align}
is dense in $L^2(-T,T;\wt H^{-1}(\mathcal{O}))$.
With the condition \eqref{W_regularity} at hand, one can directly see that 
\begin{align}\label{reg of HW}
\mathcal{H}_j \mathbf{W}_j\in \mathbb{H}^{-s}(\R^{n+1})=\mathbf{H}^{-s}(\R^{n+1}),
\end{align}
for $j=1,2$.

Suppose that there exists a function $F \in L^2(-T,T;\wt H^{-1}(\Omega))$ satisfies \eqref{F W=0 1}, then we have 
 \begin{align}\label{F W=0 3}
 \begin{split}
 		0=&\left\langle  F, \mathbf{W}_j \right\rangle_{L^2(-T,T;\wt H^{-1}(\Omega)) \times L^2(-T,T;H^1(\Omega))} \\
 		=& \left\langle  F,{\mathbf{W}}_j\right\rangle _{L^2(-T,T;H^{-1}(\R^n)) \times L^2(-T,T;H^1(\R^n))},
 \end{split}
\end{align}
where we have utilized that $F\in L^2(-T,T;\wt H^{-1}(\Omega))$ with $\mathrm{supp}(F)\subset \overline{\Omega_T}$.
In addition, there must exists a unique solution $\mathbf{v}_j\in L^2(-T,T;H^1(\R^n))$ of the backward parabolic equation 
\begin{align}\label{backward para}
	\begin{cases}
		\LC \mathcal{H}_j \RC_\ast \mathbf{v}_j =F &\text{ in }\R^n \times (-T,T), \\
		\mathbf{v}_j(x,T)=0 &\text{ in }\R^n,
	\end{cases}
\end{align}
where $\LC \mathcal{H}_j\RC_\ast=-\p_t +\mathcal{L}_j$ denotes the backward parabolic operator, for $j=1,2$.
We next analyze the regularity of the solution $\mathbf{v}_j$.

Notice that $F\in L^2(-T,T;H^{-1}(\R^n))$, then we can apply the negative fractional Laplacian $(\mathbf{Id}-\Delta)^{-1/2}=(\mathbf{Id}-\Delta_x)^{-1/2}$ to regularize the source term 
$$
\wt F:=(\mathbf{Id}-\Delta)^{-1/2}F,
$$ 
such that $\wt F\in L^2(\R^{n+1})$. One can check that $(\mathbf{Id}-\Delta)^{-1/2}$ and $\mathcal{H}_j$ are interchangeable, and apply the result as in Lemma \ref{Lem: regularity} and Remark \ref{Rmk: regularity}, then we can obtain $\widetilde{\mathbf{v}}_j\in \mathbb{H}^{1,2}(\R^{n+1})$, where $\widetilde{\mathbf{v}}_j:=(\mathbf{Id}-\Delta)^{-1/2}\mathbf{v}_j$ so that 
\begin{align}\label{reg of v_j}
	\mathbf{v}_j=(\mathbf{Id}-\Delta)^{1/2}\widetilde{\mathbf{v}}_j\in \mathbb{H}^{1,1}(\R^{n+1})\subset \mathbf{H}^s(\R^{n+1}).
\end{align}

Next, via \eqref{F W=0 3} and \eqref{backward para}, an integration by parts yields that 
\begin{align}\label{F W=0 4}
	\begin{split}
		 0=&\left\langle  F,{\mathbf{W}}_j\right\rangle _{L^2(-T,T;H^{-1}(\R^n)) \times L^2(-T,T;H^1(\R^n))}\\
		=& \int_{-T}^T \int_{\R^n} \LC \mathcal{H}_j \RC_\ast \mathbf{v}_j \cdot {\mathbf{W}}_j \, dx dt \\
		=& \int_{-T}^T \int_{\R^n} \mathbf{v}_j \LC \mathcal{H}_j {\mathbf{W}}_j\RC dxdt\\
		=& \left\langle  \mathbf{v}_j, \mathcal{H}_j {\mathbf{W}}_j\right\rangle _{\widetilde{\mathbf{H}}^{s}((\R^n)_T)\times \mathbf{H}^{-s}((\R^n)_T)},
	\end{split}
\end{align}
for $j=1,2$, where we have utilized \eqref{reg of HW} and \eqref{reg of v_j}. Via  \eqref{E_j}, one knows that 
\[
\mathcal{H}_j \mathbf{W}_j=0 \text{ in }\Omega_T,
\]
for $j=1,2$.
Combining with the preceding equality, the identity \eqref{F W=0 4} implies 
\begin{align}\label{pairing W=0}
	\left\langle  \mathbf{v}_j, \mathcal{H}_j\mathbf{W}_j\right\rangle _{\widetilde{\mathbf{H}}^s((\Omega_e)_T)\times \mathbf{H}^{-s}((\Omega_e)_T)}=0
\end{align}
for any $\mathbf{W}_j\in \mathbb{H}^{2-s}(\R^{n+1})$ solving \eqref{W_j equation}. Moreover, by utilizing the fact that $\mathcal{Y}((\Omega_e)_T)$ is also dense in $\mathbf{H}^{-s}((\Omega_e)_T)$, where $\mathcal{Y}((\Omega_e)_T)$ is defined by \eqref{Y(O)}. Thus, \eqref{pairing W=0} implies that $\mathbf{v}_j=0$ in $(\Omega_e)_T$, for $j=1,2$.

On the other hand, recalling that $\mathbf{v}_j$ is a solution of \eqref{backward para}, in particular, $\mathbf{v}_j$ satisfies $\LC \mathcal{H}_j\RC_\ast \mathbf{v}_j=0$ in $\LC \Omega_e\RC_T$. Combining with $\mathbf{v}_j=0$ in $\mathcal{O}_T$, the unique continuation property for (backward) parabolic equations yields that $\mathbf{v}_j=0$ in $\LC \Omega_e\RC_T$. To summarize, the function $\mathbf{v}_j\in L^2(-T,T;H^1_0(\Omega))$ solves 
\begin{align*}
	\begin{cases}
		\LC \mathcal{H}_j \RC_\ast \mathbf{v}_j =F &\text{ in }\Omega_T, \\
		\mathbf{v}_j=0 &\text{ on }\Sigma_T,\\
		\mathbf{v}_j(x,T)=0 &\text{ in }\Omega,
	\end{cases}
\end{align*}
and from the well-posedness for the regularity condition \eqref{reg of v_j} of $\mathbf{v}_j$, one has that $\mathbf{v}_j \in L^2(-T,T;H^1_0(\Omega))$, such that $\sigma_j\p_\nu \mathbf{v}_j\in L^2(-T,T;H^{1/2}(\Sigma))$ is well-defined for $j=1,2$. Now, since $\mathbf{v}_j=0$ in $\LC\Omega_e\RC_T$, one must have that $\sigma_j\p_\nu \mathbf{v}_j=0$ on $\Sigma_T$ for $j=1,2$. Hence, an integration by parts infers that 
\begin{align*}
	&\left\langle F,\mathbf{V}_j\right\rangle _{L^2(-T,T;\wt H^{-1}(\Omega)) \times L^2(-T,T;H^1(\Omega))}\\
	=&\int_{-T}^T \int_\Omega \LC \mathcal{H}_j \RC_\ast \mathbf{v}_j\cdot \mathbf{V}_j \, dxdt \\
	=&\int_{-T}^T \int_\Omega \mathbf{v}_j \cdot \mathcal{H}_j \mathbf{V}_j \, dxdt=0,
\end{align*}
where we used that $\mathbf{v}_j(T,x)=\mathbf{V}_j(-T,x)=0$ in $\Omega$, and $\mathcal{H}_j\mathbf{V}_j=0$ in $\Omega_T$, which proves \eqref{F W=0 2}. This completes the proof.
\end{proof}

\section{Global uniqueness and non-uniqueness}\label{Sec 5}
In the previous section, we have shown that the inverse problem for nonlocal parabolic equations and be reduced to its local counterparts. We first prove Corollary \ref{Cor: uniqueness}.

\begin{proof}[Proof of Corollary \ref{Cor: uniqueness}]
	With Theorem \ref{T1} at hand, it is known that the information of the nonlocal Cauchy data can be reduced to its local counterpart. Hence, one has that 
	$$
	\left\{v_1 |_{\Sigma_T}, \left. \sigma_1 \p_\nu v_1 \right|_{\Sigma_T} \right\} = \left\{v_2 |_{\Sigma_T}, \left. \sigma_2 \p_\nu v_2 \right|_{\Sigma_T} \right\} ,
	$$
	where $v_j\in L^2 (0,T;H^1(\Omega))$ is the weak solution of 
	\begin{align*}
		\begin{cases}
			\mathcal{H}_jv_j=0 &\text{ in }\Omega_T,\\
			v_j=f &\text{ on }\Sigma_T, \\
			v_j(-T,x)=0 &\text{ for }x\in \Omega,
		\end{cases}
	\end{align*}
for $j=1,2$. Moreover, one can apply the completeness of products of solutions to parabolic equations (for example, see \cite[Theorem 1.3]{canuto2001determining}), then we are able to conclude that $\sigma_1=\sigma_2$ in $\Omega$ as desired.
\end{proof}

Before proving Theorem \ref{T2}, let us analyze the following changing of variables, which can be regarded as the \emph{transformation optics} in the literature. Given $u\in \mathbb{H}^s(\R^{n+1})$, let $\mathbf{U}(t,\tau ,x)$ be a solution of 
\begin{align}\label{U-equation}
	\begin{cases}
		\LC \p_t +\p_\tau \RC \mathbf{U} -\nabla \cdot (\sigma \nabla  \mathbf{U})=0 & \text{ in }(-T,T)\times (0,\infty)\times \R^n,  \\
		\mathbf{U}(t,0,x)=u(t,x) &\text{ for }(t,x)\in (-T,T)\times \R^n,
	\end{cases}
\end{align}
where $\sigma$ is a globally Lipschitz continuous matrix-valued function satisfying \eqref{ellipticity condition}.

Let $\mathbf{F}:\R^n\to \R^n$ be a locally Lipschitz invertible map such that the Jacobians satisfy 
\begin{align}\label{positive Jacobian}
	\det (D\mathbf{F})(x) , \quad \det(D\mathbf{F}^{-1})(x) \geq C>0 \text{ for a.e. }x\in \R^n,
\end{align}
for some positive constant $C$.
By  direct computations, one can derive the following proposition known as the \emph{transformation optics} via the standard change of variables technique (for example, see \cite{KSVW2008cloaking}).

\begin{prop}\label{Prop: trans opt}
	$\mathbf{U}(t,\tau ,x)$ is a solution of \eqref{U-equation} if and only if $\widetilde{\mathbf{U}}(t,\tau, y)=\mathbf{U}(t,\tau,\mathbf{F}^{-1}(y))$ is a solution of 
	\begin{align}\label{U tilde-equation}
		\begin{cases}
			\mathbf{F}_\ast 1 (y)\LC \p_t +\p_\tau \RC \widetilde{\mathbf{U}} -\nabla \cdot (\mathbf{F}_\ast\sigma(y) \nabla \widetilde{\mathbf{U}})=0 & \text{ in }(-T,T)\times (0,\infty)\times \R^n,  \\
			\widetilde{\mathbf{U}} (t,0,y)=\wt u(t,y) &\text{ for }(t,y)\in (-T,T)\times \R^n,
		\end{cases}
	\end{align}
where $\wt u=u(t, \mathbf{F}^{-1}(y))$. Here the coefficients are defined by 
\begin{align*}
   \begin{cases}
   	 \mathbf{F}_\ast 1 (y) =\left. \frac{1}{\det (D\mathbf{F})(x)}\right|_{x=\mathbf{F}^{-1}(y)}, \\
   	 \mathbf{F}_\ast\sigma(y) = \left. \frac{D\mathbf{F}^T (x) \sigma (x)D\mathbf{F}(x)}{\det (D\mathbf{F})(x)}\right|_{x=\mathbf{F}^{-1}(y)} .
   \end{cases}
\end{align*}
\end{prop}
	
\begin{proof}
		The result can be seen via the standard change of variables. More precisely, by expressing \eqref{U-equation} in terms of the weak formulation, one has that 
		\begin{align}\label{weak of transfo}
			\int_{\R^n} \LC \p_t +\p_\tau \RC \mathbf{U} \varphi \, dx+ \int_{\R^n} \sigma \nabla_x \mathbf{U}\cdot \nabla_x \varphi \, dx=0,
		\end{align}
	for any $\varphi =\varphi(x) \in H^1(\R^n)$. Via the change of variable $y=\mathbf{F}(x)$ (independent of $(t,\tau)$-variables), it is not hard to see that 
	\begin{align*}
		\int_{\R^n} \sum \sigma_{ij} \frac{\p \mathbf{U}}{\p x_i}\frac{\p \varphi}{\p x_j}\, dx = \int_{\R^n} \sum \sigma_{ij} \frac{\p \mathbf{U}}{\p y_k}\frac{\p y_k}{\p x_i} \frac{\p \varphi}{\p y_\ell} \frac{\p y_\ell}{\p x_j} \det \LC \frac{\p x}{\p y}\RC dy,
	\end{align*}
where $\det \LC \frac{\p x}{\p y}\RC$ denotes the Jacobian of the change of variable $x=\mathbf{F}^{-1}(y)$.
    Inserting the above identity into \eqref{weak of transfo}, the assertion is proven.
\end{proof}

Finally, let us prove Theorem \ref{T2}.

\begin{proof}[Proof of Theorem \ref{T2}]
Let $\Omega\subset \R^n$ be a bounded domain and $W\Subset \Omega_e$ be a nonempty open set. Let $\mathbf{F}:\R^n \to \R^n$ be the Lipschitz invertible map described as before, which satisfy $\mathbf{F}:\overline{\Omega}\to \overline{\Omega}$ and \eqref{positive Jacobian}. We also assume that $\mathbf{F}(x)=x$ in $W$. Let $u\in \mathbb{H}^s(\R^{n+1})$ be a solution of $\LC \mathcal{H}_\sigma \RC ^s u=0$ in $\Omega_T$ with $u(-T, x)=0$ for $x\in \R^n$, where the nonlocal parabolic operator $\LC\mathcal{H}_\sigma \RC ^s$ can be defined by 
\begin{align*}
	\LC \mathcal{H}_\sigma\RC^s u(t,x):=-\frac{s}{\Gamma(1-s)} \int_0 ^\infty  \frac{\mathbf{U}(t,x,\tau)-u(t,x)}{\tau ^{1+s}}\, d\tau .
\end{align*}
Here $\mathcal{H}_\sigma:=\p_t -\nabla \cdot (\sigma \nabla )$ and $\mathbf{U}$ satisfies \eqref{U-equation}. Thus, adopting all notations in this section, 
\begin{align}\label{equation u}
		\LC\mathcal{H}_\sigma\RC^s u=0  \text{ in }\Omega_T \quad \text{ and}	\quad u(-T,x)=0 &\text{ in }\R^n
\end{align} 
imply that 
\begin{align*}
0= &-\frac{s}{\Gamma(1-s)}\int_0^\infty \frac{\mathbf{U}(t,\tau ,x)-u(t,x)}{\tau^{1+s}}\, d\tau  \\
=& -\frac{s}{\Gamma(1-s)}\int_0^\infty \frac{\widetilde{\mathbf{U}}(t,\tau ,y)-\wt u(t,y)}{\tau^{1+s}}\, d\tau , \quad \text{ for }(t,x), \ (t,y)\in \Omega_T
\end{align*}
where $\widetilde{\mathbf{U}}(t,\tau ,y)$ is a solution to \eqref{U tilde-equation}. Meanwhile, $\wt u(-T,y)=0$, which yields that $\wt u\in \mathbb{H}^s(\R^{n+1})$ is a solution to 
\begin{align}\label{equation tilde u}
	\LC\mathcal{H}_{\mathbf{F}_\ast \sigma}\RC^s \wt u=0  \text{ in }\Omega_T \quad \text{ and}	\quad \wt u(-T,y)=0 &\text{ in }\R^n.
\end{align}

On the other hand, in viewing of the nonlocal Cauchy data, we can derive that $u(t,\cdot )=\wt u(t,\cdot)$ in $W_T$ and $\mathbf{U}(t,\tau,\cdot )=\widetilde{\mathbf{U}}(t,\tau ,\cdot)$ in $W$, for $(t,\tau)\in (-T,T)\times (0,\infty)$, then 
\begin{align*}
	\LC\mathcal{H}_\sigma\RC^s u(t,x)=&-\frac{s}{\Gamma(1-s)}\int_0^\infty \frac{\mathbf{U}(t,\tau ,x)-u(t,x)}{\tau^{1+s}}\, d\tau \\
	=& -\frac{s}{\Gamma(1-s)}\int_0^\infty \frac{\widetilde{\mathbf{U}}(t,\tau ,y)-\wt u(t,y)}{\tau^{1+s}}\, d\tau \\
	=&	\LC\mathcal{H}_{\mathbf{F}_\ast \sigma}\RC^s \wt u \quad  \text{ in } W_T.
\end{align*}
The preceding derivation yields that there are two different matrix-valued functions $\sigma$ and $\mathbf{F}_\ast \sigma$ can generate the same exterior Cauchy data 
\begin{align*}
	\left\{ u|_{W_T}, \left. \LC\mathcal{H}_\sigma \RC^s u \right|_{W_T} \right\}=	\left\{ \wt u|_{W_T}, \left.\LC\mathcal{H}_{\mathbf{F}_\ast\sigma} \RC^s \wt u \right|_{W_T} \right\},
\end{align*}
where $u$ and $\wt u $ are solutions to \eqref{equation u} and \eqref{equation tilde u}, respectively. This completes the proof.
\end{proof}

\vskip0.5cm

\noindent\textbf{Acknowledgments.} 
Y.-H. Lin is partly supported by MOST 111-2628-M-A49-002. G. Uhlmann was partially
supported by NSF, a Walker Professorship at University of Washington and a Si-Yuan
Professorship at Institute for Advanced Study, Hong Kong University of Science and Technology.

\bibliographystyle{alpha}
\bibliography{ref}

\end{document}